\documentclass[11pt,oneside]{amsart}
\usepackage{amsmath,amsfonts,amsgen,amstext,amsbsy,amsthm, amssymb}
\usepackage{multirow}
\usepackage{bbold}
\usepackage{enumitem}
\usepackage[T1]{fontenc}
\usepackage[colorlinks=true,hyperindex=true]{hyperref}
\usepackage{tikz}
\usetikzlibrary{calc}
\usepackage{verbatim}
\usepackage{mathrsfs}
\usepackage{tkz-fct} 
\usepackage{pgfplots}
\makeatletter

\usepackage{setspace}
\numberwithin{equation}{section}

\newcounter{smallarabics}

\newcounter{smallroman}

\newtheorem{theoreme}{theorem }[section]
\newtheorem{theorem}[theoreme]{Theorem}
\newtheorem{proposition}[theoreme]{Proposition}
\newtheorem{Lemma}[theoreme]{Lemma}

\newtheorem{remark}{Remark}[section]
\newtheorem{example}[theoreme]{Example}

\newcommand{\vr}{\varphi_r}
\newcommand{\vl}{\varphi_{\ell}}

\renewcommand{\i}{\mathrm{i}}
\newcommand{\tA}{\tilde{A}}

\newcommand\nn\nonumber
\renewcommand\leq\varleq
\renewcommand\geq\vargeq
\newcommand{\sign}{\mathrm{sign}}
\renewcommand{\star}{*}
\newcommand{\1}{\mathbb 1}

\newcommand{\T}{\tau}

 \newcommand{\R}{\mathbb{R}}
 \newcommand{\N}{\mathbb{N}}
\newcommand{\Z}{\mathbb{Z}} \newcommand{\C}{\mathbb{C}}
\newcommand{\D}{\mathcal{D}} 
 \renewcommand{\H}{\mathcal{H}}

\newcommand{\grad}{\nabla}

\DeclareMathOperator{\supp}{supp}

\makeatletter
      \def\@setcopyright{}
      \def\serieslogo@{}
      \makeatother
\begin{document}

\author{Mandich, Marc-Adrien}
   \address{Institut de Math\'ematiques de Bordeaux, 351 cours de la Lib\'{e}ration, F 33405 Talence, France}
   \email{marc-adrien.mandich@u-bordeaux.fr}


   \title[Sub-exponential decay for discrete Schr\"odinger operators]{Sub-exponential decay of eigenfunctions for some discrete Schr\"odinger operators}

\begin{abstract}
Following the method of Froese and Herbst, we show for a class of potentials $V$ that an eigenfunction $\psi$ with eigenvalue $E$ of the multi-dimensional discrete Schr\"odinger operator $H = \Delta + V$ on $\Z^d$ decays sub-exponentially whenever the Mourre estimate holds at $E$. In the one-dimensional case we further show that this eigenfunction decays exponentially with a rate at least of $\text{cosh}^{-1}((E-2)/(\theta_E-2))$, where $\theta_E$ is the nearest threshold of $H$ located between $E$ and $2$. A consequence of the latter result is the absence of eigenvalues between $2$ and the nearest thresholds above and below this value. The method of Combes-Thomas is also reviewed for the discrete Schr\"odinger operators.
\end{abstract}

%



   \maketitle

\section{Introduction}

The analysis of the decay rate of eigenfunctions of Schr\"odinger operators goes back to the famous works of Slaggie and Wichmann \cite{SW}, Agmon \cite{A1}, and Combes and Thomas \cite{CT}. Their results showed that eigenfunctions corresponding to eigenvalues located outside the essential spectrum decay exponentially. Subsequently, Froese and Herbst \cite{FH}, but also \cite{FHHO1} and \cite{FHHO2}, investigated the decay of eigenfunctions corresponding to eigenvalues located in the essential spectrum of Schr\"odinger operators. They showed that eigenfunctions of the continuous Schr\"odinger operator on $\R^n$ decay exponentially at non-threshold energies for a large class of potentials. Since their pioneering work a solid literature has grown using these ideas. For example, these ideas have been applied to Schr\"odinger operators on manifolds \cite{V}, Schr\"odinger operators in PDE's \cite{HS}, and self-adjoint operators in Mourre theory \cite{FMS}. This short list is by no means complete. The question however does not seem to have been investigated for the discrete Schr\"odinger operator on the lattice and constitutes the subject of this paper. For completeness and convenience, this paper will also review the Combes-Thomas method for the discrete Schr\"odinger operators. A nice historical review on the exponential decay of eigenfunctions is done by Hislop in \cite{Hi}.

We now describe the mathematical setup of the article. The configuration space is the multi-dimensional lattice $\Z^d$ for some integer $d\geqslant 1$. For a multi-index $n=(n_1,...,n_d) \in \Z^d$, we set $|n|^2 := n_1^2 + ...+n_d^2$. Consider the complex Hilbert space $\mathcal{H} := \ell^2(\Z^d)$ of square summable sequences $(u(n))_{n \in \Z^d}$. The discrete Schr\"odinger operator acting on $\mathcal{H}$ is 
\begin{equation}
H := \Delta + V,
\end{equation}
where $\Delta$ is the non-negative discrete Laplacian defined by 
\begin{equation*}
(\Delta u)(n) := \sum _{\substack{m \in \Z^d, \\ |n-m|=1}} (u(n)-u(m)), \quad \text{for all} \ n \in \Z^d, u \in \mathcal{H},
\end{equation*}
and $V$ is a multiplication operator by a bounded real-valued sequence $(V(n))_{n \in \Z^d}$. It is common knowledge that the spectrum of $\Delta$, denoted $\sigma(\Delta)$, is purely absolutely continuous and equals $[0,4d]$. Define for $(\alpha, \gamma) \in [0,\infty)\times [0,1]$ the operator of multiplication on $\mathcal{H}$ by
\begin{equation}
\label{Exponential}
\vartheta_{\alpha,\gamma} := \exp \left(\alpha \left(1+|n|^2\right)^{\gamma/2}\right), \quad  \text{with domain}
\end{equation}
\begin{equation*}
\mathcal{D}(\vartheta_{\alpha,\gamma}) := \left\{ u \in \mathcal{H} : \sum_{n \in \Z^d} \exp \left(2\alpha\left(1+|n|^2\right)^{\gamma/2}\right)|u(n)|^2 < +\infty \right\}.
\end{equation*}
In this manuscript, we will say that $\psi \in \mathcal{H}$ decays sub-exponentially (resp.\ exponentially)  if $\psi \in \mathcal{D}(\vartheta_{\alpha,\gamma})$ for some $\gamma < 1$ (resp.\ for $\gamma =1$) and some $\alpha > 0$. Write $\vartheta_{\alpha} := \vartheta_{\alpha,1}$. We begin with a well-known fact and formulate a version of the main result of Combes and Thomas in the context of multi-dimensional discrete Schr\"odinger operators:
\begin{theorem} 
\label{CTmethod}
Let $(V(n))_{n \in \Z^d}$ be a bounded sequence. Suppose that $H \psi = E \psi$, with $\psi \in \H$ and $E \in \R \setminus \sigma(\Delta) = (-\infty,0) \cup (4d, +\infty)$. If $\limsup_{|n| \to +\infty} |V(n)|  < \mathrm{dist}(\sigma(\Delta),E)$, then there exists $\nu >0$ depending on $\mathrm{dist}(\sigma(\Delta),E)$ such that for all $\alpha \in [0,\nu)$, $\psi \in \D(\vartheta_{\alpha})$. 
\end{theorem}
\begin{remark}
We recall that in the discrete setting, a multiplication operator $V$ is compact if and only if $\lim_{| n| \to + \infty} V(n) =0$. If $V$ is compact, then $0 = \limsup_{|n| \to +\infty} |V(n)|  < \mathrm{dist}(\sigma(\Delta),E)$ is automatically verified and also $\sigma_{\mathrm{ess}}(H) = \sigma_{\mathrm{ess}}(\Delta) = \sigma(\Delta)$ by Weyl's Theorem. So in this case, Theorem \ref{CTmethod} is indeed proving the exponential decay of the eigenfunction $\psi$ when the eigenvalue $E$ is located outside the essential spectrum of $H$. 
\end{remark}
The advantage of the perturbative method of Combes-Thomas is that it yields exponential decay of eigenfunctions with a convenient and explicit geometric bound under rather general assumptions for the potential. Another big plus is that it is easy to implement in many different scenarios. The drawback however is that it does not work if the eigenvalue $E$ belongs to the spectrum of the free operator $\Delta$. In addition to the aforementioned references, we refer to \cite{BCH} for an improved Combes-Thomas method with optimal exponential bounds. 

The method of Froese and Herbst does not exploit a condition like $\mathrm{dist}(\sigma(\Delta),E) >0$, but rather a Mourre estimate, which is a local positivity condition on the commutator between $H$ and some appropriate conjugate operator. The article is largely devoted to the study of this method. Before presenting the results, we elaborate on the Mourre estimate, the key relation in the theory developed by Mourre \cite{Mo}. We refer to \cite{ABG} and references therein for a thorough overview of the improved theory. The position operator $N=(N_1,...,N_d)$ is defined by 
\begin{equation}
(N_iu)(n) := n_i u(n), \quad \mathcal{D}(N_i) := \left\{ u\in \ell^2(\Z^d) : \sum_{n \in \Z^d} |n_i u(n)|^2 < +\infty \right\},
\end{equation}
and the shift operators $S_i$ and $S_i^*$ to the right and to left respectively act on $\mathcal{H}$ by
\begin{equation}
(S_i u)(n) := u(n_1,...,n_i-1,...,n_d), \quad \text{for all} \ n \in \Z^d \ \text{and} \ u \in \mathcal{H},
\end{equation}
and correspondingly for $S_i^*$. We note that the Laplacian may alternatively be written as $\Delta = \sum_{i=1}^d (2-S_i^*-S_i)$. The conjugate operator to $H$ that is used in this manuscript is the discrete version of the so-called generator of dilations. We denote it by $A$ and it is the closure of the operator $A_0$ given by
\begin{equation}
A_0 := \i\sum \limits_{i = 1}^d 2^{-1}(S_i^*+S_i) - (S_i^*-S_i)N_i = -\i\sum \limits_{i=1}^d 2^{-1}(S_i^*+S_i) + N_i(S_i^*-S_i)
\end{equation}
with domain $\mathcal{D}(A_0) = \ell_0(\Z^d)$, the collection of sequences with compact support. It is well-known that $A$ is a self-adjoint operator, see e.g.\ \cite{GGo}. Let $T$ be an arbitrary bounded self-adjoint operator on $\mathcal{H}$. If the form 
\begin{equation*}
(u,v) \mapsto \langle u, [T,A] v \rangle := \langle Tu, Av \rangle - \langle Au, Tv \rangle 
\end{equation*}
defined on $\mathcal{D}(A) \times \mathcal{D}(A)$ extends to a bounded form on $\mathcal{H} \times \mathcal{H}$, we denote by $[T,A]_{\circ}$ the bounded operator extending the form, and say that $T$ is of class $C^1(A)$, cf.\ \cite{ABG}[Lemma 6.2.9]. We refer the reader to  \cite{ABG}[Theorem 6.2.10] for equivalent definitions of this class. We have that 
\begin{equation}
\label{CommuteBaby}
[\Delta,\i A]_{\circ} = \sum_{i=1}^d \Delta_i(4-\Delta_i) = \sum_{i=1}^d (2-(S_i^*)^2 - (S_i)^2)
\end{equation} 
and this is a non-negative operator. We must also discuss the commutator between the potential $V$ and $A$. To this end, denote by $\T_iV$ and $\T_i^*V$ the operators of multiplication by the shifted sequence $(V(n))_{n \in \Z^d}$ to the right and left respectively on the i$^{\text{th}}$ coordinate, namely
\begin{equation*}
[(\T_iV)u](n) := V(n_1,...,n_i-1,...,n_d)u(n), \quad \forall n \in \Z^d, u \in \mathcal{H}, \text{and} \ i=1,...,d,
\end{equation*}
and correspondingly for $\T_i^*V$. The commutator between $V$ and $A$ is given by
\begin{equation}
\label{commutatorV}
\langle u, [V,\i A] v \rangle = \sum_{i=1}^d \langle u, [(2^{-1}-N_i)(V-\T_iV)S_i + (2^{-1}+N_i)(V-\T_i^*V)S_i^*]v\rangle, \quad  \forall u,v \in \ell_0(\Z^d).
\end{equation}
Assuming $V$ to be bounded, note that $[V,\i A]_{\circ}$ exists if and only if Hypothesis 1 stated below holds. Assuming $[H,\i A]_{\circ}$ to exist, we say that the Mourre estimate holds at $\lambda \in \R$ if there exists an open interval $\Sigma$ containing $\lambda$, a constant $c >0$ and a compact operator $K$ such that 
\begin{equation}
\label{Mourre}
E_{\Sigma}(H)[H,\i A]_{\circ} E_{\Sigma}(H) \geqslant c E_{\Sigma}(H) +K,
\end{equation}
in the form sense on $\mathcal{H} \times \mathcal{H}$. Here $E_{\Sigma}(H)$ is the spectral projector of $H$ onto the interval $\Sigma$. Denote $\Theta(H)$ the set of points where a Mourre estimate \eqref{Mourre} holds for $H$ with respect to $A$. In other words, $\R \setminus \Theta(H)$ is the set of \emph{thresholds} of $H$. 
In addition to $V$ bounded, two hypotheses on the potential appear in this manuscript: \\
\textbf{Hypothesis 1}: The potential $V$ satisfies 
\begin{equation}
\label{Hyp1}
\max_{1 \leqslant i \leqslant d} \sup_{n \in \Z^d} \ |n_i (V-\T_iV)(n))| < +\infty.  
\end{equation}
\textbf{Hypothesis 2}: $V$ is compact, i.e. 
\begin{equation}
\label{Hyp2}
V(n) \to 0, \quad \text{as} \ |n| \to +\infty.
\end{equation}
The main result of the paper concerning the one-dimensional operator $H$ is: 
\begin{theorem} 
\label{HunGun1}
Assume Hypotheses 1 and 2, and $d=1$. If $H\psi = E \psi$ with $\psi \in \ell^2(\Z)$, then if
\begin{equation}
\theta_E := \begin{cases}
\sup \left\{ 2 + (E-2)/\cosh{\alpha}  : \alpha \geqslant 0 \ \text{and} \ \psi \in \mathcal{D}(\vartheta_{\alpha}) \right \}, &\text{for} \ \ E < 2\\
\inf \left\{2 + (E-2)/\cosh{\alpha}  : \alpha \geqslant 0 \ \text{and} \ \psi \in \mathcal{D}(\vartheta_{\alpha}) \right\}, &\text{for} \ \ E > 2,
\end{cases} 
\end{equation}
one has that either $\theta_E \in \R \setminus \Theta(H)$ or $\theta_E = 2$. If $E = 2$, the statement is that either $ \psi \in \mathcal{D}(\vartheta_{\alpha})$ for all $\alpha \geqslant 0$ or $2 \in \R \setminus \Theta(H)$. Moreover, if $\psi \in \mathcal{D}(\vartheta_{\alpha})$ for all $\alpha \geqslant 0$, then $\psi = 0$. 
\end{theorem}
\begin{remark}
\label{dldl}
The function $\R^+ \ni \alpha \mapsto \theta_E(\alpha) := 2 + (E-2)/\cosh(\alpha) \in [E,2)$ is increasing to two when $E < 2$ so that $E \leqslant \theta_E \leqslant 2$, whereas the function is decreasing to two when $E>2$ in which case $E \geqslant \theta_E \geqslant 2$. This function is graphed in Figure \ref{Graph} for four different values of $E$.
\end{remark}

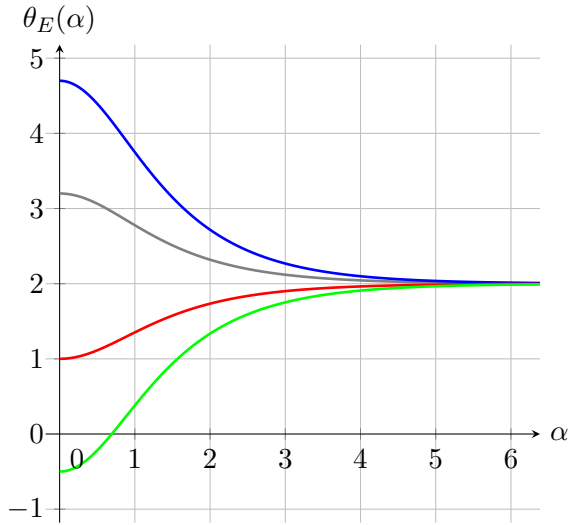
\begin{figure} [h]
\begin{tikzpicture}[domain=0.01:6]
\begin{axis}
[grid = major, 
clip = true, 
clip mode=individual, 
axis x line = middle, 
axis y line = middle, 
xlabel={$\alpha$}, 
xlabel style={at=(current axis.right of origin), anchor=west}, 
ylabel={$\theta_E(\alpha)$}, 
ylabel style={at=(current axis.above origin), anchor=south}, 
domain = 0.01:6, 
xmin = 0, 
xmax = 6.2, 
y=1cm,
x=1cm,
xtick={0,...,6},
ytick={-1,0,1,2,3,4,5},
enlarge y limits={rel=0.03}, 
enlarge x limits={rel=0.03}, 
ymin = -1, 
ymax = 5, 
after end axis/.code={\path (axis cs:0,0) node [anchor=north west,yshift=-0.075cm] {0} node [anchor=east,xshift=-0.075cm] {0};}]

\addplot[color=gray,line width = 1.0 pt, samples=400,domain=0:10] ({x},{2+(3.2-2)/cosh(x)});
\addplot[color=blue,line width = 1.0 pt, samples=400,domain=0:10] ({x},{2+(4.7-2)/cosh(x)});
\addplot[color=red,line width = 1.0 pt, samples=400,domain=0:10] ({x},{2+(1-2)/cosh(x)});
\addplot[color=green,line width = 1.0 pt, samples=400,domain=0:10] ({x},{2+(-0.5-2)/cosh(x)});
\end{axis}
\end{tikzpicture}
\caption{Graph of $\theta_E(\alpha) = 2 + (E-2)/\cosh(\alpha)$ for four different values of $E$.}
\label{Graph}
\end{figure}

If $E$ is both an eigenvalue and a threshold, Theorem \ref{HunGun1} does not give any information about the rate of decay of the corresponding eigenfunction, whereas if $E$ is not a threshold, the corresponding eigenfunction decays at a rate at least of $\cosh^{-1}((E-2)/(\theta_E-2))$. As in the continuous operator setting, the possibility of $\psi \in \mathcal{D}(\vartheta_{\alpha})$ for all $\alpha \geqslant 0$ can be eliminated. The last part of Theorem \ref{HunGun1} implies the absence of eigenvalues in the middle of the band $[0,4]$, more precisely between $2$ and the nearest thresholds above and below this value. 

The study of the absence of positive eigenvalues for Schr\"odinger operators has a long history. For continuous Schr\"odinger operators, it was shown in the sixties in articles by Kato \cite{K2}, Simon \cite{Si1} and Agmon \cite{A2} that the multi-dimensional operator $-\Delta+V_1 + V_2$ has no eigenvalues in $[0,+\infty)$ whenever $\lim _{|x| \to +\infty} |x| |V_1(x)|=0$ and $\lim _{|x| \to +\infty} |(x\cdot \grad) V_2(x)| =0$. In fact, the method of Froese and Herbst allows to extend this result to $N$-body Hamiltonians, see \cite[Theorem 4.19]{CFKS}. So, if the discrete case were to resemble the continuous case, it is not unreasonable to expect the multi-dimensional operator $\Delta + V$ to have no eigenvalues in $(0,4d)$ whenever $|n_i(V(n) - \tau_i V(n))| \to 0$ as $|n| \to + \infty$. A one-dimensional result pointing in this direction is the following. It actually comes as a corollary of Theorem \ref{HunGun1}.

\begin{theorem}
\label{CorollaryHun}
Let $d=1$. Suppose that $V$ satisfies $\lim_{|n| \to +\infty} |n| |V(n) - V(n-1)| = 0$ and $\lim _{|n| \to +\infty} |V(n)| = 0$. Then $H := \Delta + V$ has no eigenvalues in $(0,4)$.
\end{theorem}
\begin{proof}
First, if $|n(V(n)-V(n-1))| \to 0$, we see from \eqref{commutatorV} that $[V,\i A]_{\circ}$ is not only a bounded operator but also compact. It follows by \cite[Proposition 2.1]{GMa} that $V \in C^1_{\text{u}}(A)$. Let $\mathcal{B}(\mathcal{H})$ denote the bounded operators on $\mathcal{H}$. We recall that a bounded operator $T$ belongs to the $C^1_{\text{u}}(A)$ class if the map $\R \mapsto e^{-\i tA} T e^{\i tA}$ is of class $C^1(\R; \mathcal{B}(\mathcal{H}))$, with $\mathcal{B}(\mathcal{H})$ endowed with the norm operator topology. It is well-known that $\Delta$ is of class $C^1_{\text{u}}(A)$, see e.g.\ \cite{Man}. We then apply \cite[Theorem 7.2.9]{ABG} to conclude that $\Theta(H)=\Theta(\Delta)= (0,4)$. Here $\Theta(\Delta)$ denotes the set of points where a Mourre estimate holds for $\Delta$ with respect to $A$, and $\Theta(\Delta)= (0,4)$ is a direct consequence of\eqref{CommuteBaby}. Since $H$ does not have any thresholds in $(0,4)$, it must be that $H$ has no eigenvalues in this interval, by Theorem \ref{HunGun1}. 
\qed
\end{proof}

This is very much related to Remling's optimal result \cite{R}, that if $\lim_{|n| \to +\infty} |n| |V(n)| = 0$, then the spectrum of the one-dimensional discrete operator $\Delta+V$ is purely absolutely continuous on $(0,4)$. Of course, Remling's result is stronger than that of Theorem \ref{CorollaryHun}, but the assumptions are also stronger. Also related is a one-dimensional discrete version of Weidmann's Theorem proven in \cite{Si2}, namely if $V$ is compact and of bounded variation, then the spectrum of $\Delta +V$ is purely absolutely continuous on $(0,4)$. Finally, another interesting result is that of \cite{JS} where it is shown that the spectrum of the half-line discrete Schr\"odinger operator $\Delta +W+V$ is purely absolutely continuous on $(0,4) \setminus \{2\pm 2\cos(k/2)\}$, where $W(n)=q \sin(kn)/n^{\beta}$ with $q,k \in \R, \beta \in (1/2,1]$ and $(V(n)) \in \ell^1(\Z_+)$. Note that Theorem \ref{HunGun1} is in conformity with their example when $\beta=1$ and $V\equiv 0$. In the same spirit, we provide a simple application of Theorem \ref{CorollaryHun}:
\begin{proposition} 
\label{WVNJM}
Let $d=1$ and $W(n) := q \sin (k |n|^{\alpha}) / |n| ^{\beta}$ be a Wigner-von Neumann potential, with $q,k \in \R$. Then for $\beta > \alpha >0$, $\sigma_{\rm{ess}}(\Delta + W) = [0,4]$ and $(0,4)$ is void of eigenvalues.
\end{proposition}
An analogous result for continuous Schr\"odinger operators is obtained and thoroughly discussed in \cite{JM}, and is also inspired from \cite{FH}. We now turn to the multi-dimensional discrete Schr\"odinger operators. The main result concerning these is:
\begin{theorem}
\label{sunnyHunny}
Let $d \geqslant 1$. Suppose that Hypothesis 1 holds for the potential $V$. If $H\psi = E\psi$ with $\psi \in \ell^2(\Z^d)$ and $E \in \Theta(H)$, then $\psi \in \mathcal{D}(\vartheta_{\alpha,\gamma})$ for all $(\alpha,\gamma) \in [0,\infty)\times[0,2/3)$. 
\end{theorem}
Although Theorem \ref{sunnyHunny} does not yield exponential decay of eigenfunctions at non-threshold energies as in the continuous operator case, the result is still useful for applications in Mourre theory. It appears that the method of Froese and Herbst adapts quite well for the one-dimensional discrete operator; however, there seems to be a non-trivial difference between the dimensions $d \geqslant 2$ and $d=1$ in the discrete setting as far as the method is concerned. The exponential decay of eigenfunctions at non-threshold energies in higher dimensions therefore remains an open question because our proof does not attain it. Yet an indication it may occur comes from the Combes-Thomas method presented above. 

On the one hand, if $E$ belongs to the discrete spectrum of $H$, then for any interval $\Sigma$ containing $E$ and located outside the essential spectrum of $H$, $E_{\Sigma}(H)$ is simply a finite rank eigenprojection and so the Mourre estimate holds by default, both sides of \eqref{Mourre} being compact operators. So under Hypothesis 1 only, the corresponding eigenfunction decays sub-exponentially according to Theorem \ref{sunnyHunny}. In this case, the Combes-Thomas method is clearly superior. On the other hand, the Mourre estimate typically holds above the essential spectrum of $H$. So Theorem \ref{sunnyHunny} is able to characterize the decay of eigenfunctions for non-threshold eigenvalues embedded in the essential spectrum, if \textit{any} exist. We emphasize the last point, because to our knowledge there is no example of a Schr\"odinger operator with a non-threshold embedded eigenvalue. What is certainly known however is the existence of operators with a threshold embedded eigenvalue, the Wigner-von Neumann operator being the classical illustration of it, see e.g.\ \cite{RS4}. 

Let us provide an example of a discrete Wigner-von Neumann type operator $H$ that has an eigenvalue embedded in its essential spectrum. An eigenvector for this eigenvalue will be given explicitly. Here's how Theorem \ref{sunnyHunny} turns out to be useful: as the eigenvector will have slow decay at infinity, we infer that the eigenvalue is a threshold, in the sense that no Mourre estimate holds for the pair of self-adjoint operators $(H,A)$ above any interval containing this value. Our example and approach is inspired from the one that appears in \cite[Section XIII.13, Example 1]{RS4}. 

\begin{proposition} 
\label{WVN}
For given $k_1,...,k_d \in (0,\pi)$, let $(t_{k_i})_{i=1}^d$ be real numbers such that 
\[t_{k_i} + \sin(2k_i)n_i-\sin(2k_i n_i) \neq 0, \quad  \text{for all} \ n_i \in \Z.\] 
Then there exists an oscillating potential $V$ on $\Z^d$ that has the asymptotic behavior
\[ V(n_1,...,n_d) = \sum_{i=1}^d -\frac{4\sin(k_i)\sin(2k_in_i)}{n_i}+ O_{k_i,t_{k_i}}(n_i^{-2})\]
and such that $E := 2d - \sum_{i=1}^d 2\cos(k_i)$ is both a threshold and an eigenvalue for $H := \Delta+V$, with eigenvector $\psi(n_1,...,n_d) = \prod_{i=1}^d \sin(k_i n_i) [t_{k_i} + \sin(2k_i)n_i-\sin(2k_i n_i)]^{-1}$ belonging to $\ell^2(\Z^d)$. Moreover, $E \in [0,4d] \subset \sigma_{\mathrm{ess}}(H)$.
\end{proposition} 

The exact expression of the potential $V$ is given in the proof. By the notation $O_{k_i,t_{k_i}}(n_i^{-2})$, we mean that this decaying term depends on the choice of $k_i$ and $t_{k_i}$. It is interesting to further note that the eigenvector $\psi$ does not belong to the domain of $A$, for $\left(N_i(S_i^*-S_i)\psi \right)(n_1,...,n_d)$ does not go to zero as $|n_i | \to +\infty$. To further motivate Theorem \ref{sunnyHunny}, let us give another application to discrete Wigner-von Neumann operators. 

\begin{example}[from \cite{Man}] Let $W$ be the discrete Wigner-von Neumann potential given by 
\begin{equation*}
(Wu)(n) = W(n)u(n) := \frac{q\sin(k(n_1+...+n_d))}{|n|}u(n), \quad \forall n \in \Z^d, u \in \mathcal{H},
\end{equation*}
for some $(q,k) \in \R\times (-\pi,\pi)$, and let $V$ be a multiplication operator satisfying for some $\rho >0$,
\begin{equation*}
\sup_{n\in \Z^d} \ \langle n \rangle ^{\rho} |V(n)| < \infty, \quad \text{and}  \quad   
\max_{1 \leqslant i \leqslant d} \sup_{n \in \Z^d} \langle n \rangle ^{\rho}|n_i| |(V-\T_iV)(n)| < +\infty.    
\end{equation*}
Here $\langle n \rangle := \sqrt{1+|n|^2}$. Let $H := \Delta + W +V$ be the Schr\"odinger operator on $\mathcal{H}$, and let  $P$ and $P^{\perp}$ respectively denote the spectral projectors onto the pure point subspace of $H$ and its complement. Let $E(k) := 4-4\cdot \sign(k)\cos(k/2)$, and consider the sets 
\begin{align*}
\mu(H) &:= (0,4) \setminus \{2\pm 2\cos(k/2)\}, \quad \text{for} \ d=1, \\
\mu(H) &:= (0,E(k))\cup(4d-E(k),4d), \quad  \text{for} \ d\geqslant 2.
\end{align*}
By combining Theorem \ref{sunnyHunny} with \cite[Theorem 1.1]{Man}, one can remove the abstract assumption $\ker(H-E) \subset \mathcal{D}(A)$ that appears in the latter Theorem; and for the one-dimensional result, we can use the stronger result of Theorem \ref{HunGun1}. We get the following improved result:
\begin{theorem}
We have that $\mu(H) \subset \Theta(H)$. For all $E \in \mu(H)$ there is an open interval $\Sigma$ containing $E$ such that for all $s>1/2$ and all compact intervals $\Sigma' \subset \Sigma$, the reduced limiting absorption principle for $H$ holds for with respect to $(\Sigma',s,A)$, that is,
\begin{equation*}
\sup \limits_{x \in \Sigma', y \neq 0} \|\langle A \rangle ^{-s} (H-x-\i y)^{-1} P^{\perp}\langle A \rangle ^{-s}\| < \infty.
\end{equation*}
In particular, the spectrum of $H$ is purely absolutely continuous on $\Sigma'$ whenever $P=0$ on $\Sigma'$, and for $d=1$, $H$ does not have any eigenvalues in the interval $(2-2\cos(k/2),2+2\cos(k/2))$.
\end{theorem}
\end{example}

From a perspective of Mourre theory and in an abstract setting, an area of research is to show that the eigenfunction $\psi \in \mathcal{D}(A^n)$ for some $n \geqslant 1$. The first results of this kind were obtained in \cite{Ca} and \cite{CGH}, where it was shown that if $H\psi = E \psi$ with $E$ embedded in the continuous spectrum of $H$, and the iterated commutators $\text{ad}^{k}_A(H)$ are bounded for $k=1,...,\nu$ together with appropriate domain conditions being satisfied by $H$ and $A$, then $\psi \in \mathcal{D}(A^n)$ for all $n \geqslant 0$ satisfying $n+2 \leqslant \nu$, whenever the Mourre estimate holds at $E$. Here $A$ is the conjugate operator to the Hamiltonian $H$ in the abstract framework, and the iterated commutators are defined by $\text{ad}^1_A(H) := [H, \i A]_{\circ}$ and $\text{ad}^k_A(H) := [\text{ad}^{k-1}_A(H), \i A]_{\circ}$. So in the simplest case, one would obtain $\psi \in \mathcal{D}(A)$ provided $\text{ad}^3_A(H)$ exists. Then in \cite{FMS}, the authors reduce by one, from $n+2$ to $n+1$ the number of commutators that need to be bounded in order to obtain $\psi \in \mathcal{D}(A^n)$, and show that the result is optimal. In counterpart of these abstract results, we should point out that in the framework of Schr\"odinger operators, minimal hypotheses yield much stronger results. Indeed, a direct consequence of Theorem \ref{sunnyHunny} is that $\psi \in \mathcal{D}(A^n)$ for all $n\geqslant 0$ assuming only $[H,\i A]_{\circ}$ bounded. 

Finally, we point out that the notion of the $C^1(A)$ class of operators also exists for unbounded operators. It appears to us that the results of this paper could also apply to Schr\"odinger operators with unbounded potentials satisfying the $C^1(A)$ condition. A simple criterion to check if the potential belongs to this class is given in \cite{GMo}[Lemma A.2]. This criterion is straightforward to verify in the setting of this paper. It is however doubtful to us if the generalization of the result to unbounded potentials is significant.

The plan of the paper is as follows: in Section \ref{CTsection}, we provide a proof of Theorem \ref{CTmethod} for the reader's convenience. Section \ref{MultiD} is devoted to the proof of the main result for the multi-dimensional Schr\"odinger operator, namely Theorem \ref{sunnyHunny}. In Section \ref{SectionWVN}, we prove Proposition \ref{WVN}. In Section \ref{hut89}, we further develop the method of Section \ref{MultiD} in the case of the one-dimensional operator, and prove Theorem \ref{HunGun1}. Finally Section \ref{appendix} is the Appendix and contains a long technical calculation proving a key relation required for both Sections \ref{MultiD} and \ref{hut89}. \\

\noindent \textbf{Acknowledgments:}  It is a pleasure to thank my thesis director Sylvain Gol\'enia for his numerous useful comments and advice, and also Thierry Jecko and Milivoje Lukic for enlightening conversations. I thank the anonymous referee for a very helpful and constructive report. I am grateful to the University of Bordeaux for funding my studies.






\section{The method of Combes-Thomas: Proof of Theorem \ref{CTmethod}}
\label{CTsection}
We follow the approach given in \cite{Hi} and to a lesser extent \cite{BCH}. We point out that the Combes-Thomas method typically involves techniques of analytic continuation which require some care if the operators are unbounded, see e.g.\  \cite[Section XII.2]{RS4}.  However, since all operators are bounded in this setting, things are simpler. Let $\mathcal{B}(\mathcal{H})$ be the bounded operators on $\mathcal{H}$, and let $\rho = \rho(n) := \sqrt{1+|n|^2}$, $n\in \Z^d$. First we need an estimate:
\begin{proposition} 
\label{propYouGo}
Let $V$ be any bounded real-valued potential, and denote $T := \Delta + V$. Then $\C \ni \lambda \mapsto T(\lambda) := e^{\i  \lambda \rho} T e^{-\i \lambda \rho} \in \mathcal{B}(\mathcal{H})$ is an analytic map. If $E \in \R \setminus \sigma(T)$, then for $\lambda$ satisfying 
\begin{equation}
\label{condition}
\frac{2d \cdot e^{|\lambda |} |\lambda | }{\mathrm{dist}(\sigma(T),E)}< \frac{1}{2}, 
\end{equation}
\begin{equation} 
\label{result222}
\|(T(\lambda) -E)^{-1} \| \leqslant 2 / \mathrm{dist}(\sigma(T),E).
\end{equation}
\end{proposition}
\begin{proof}
A first calculation gives that
\[ T(\lambda) := e^{\i  \lambda \rho } T e^{-\i  \lambda \rho} = T +D(\lambda), \]
where 
\begin{equation*} D(\lambda) :=   \sum_{i =1} ^d \left( 1 - e^{\i \lambda (\rho - \tau_i \rho)} \right) S_i + \left(1-e^{-\i  \lambda (\rho - \tau_i^* \rho)} \right) S^* _i.
\end{equation*}
By the Mean Value Theorem, $| \rho - \tau_i \rho |$ and $| \rho - \tau_i^* \rho |$ are bounded above by one. Also, $\| S_i \| = \| S_i^*\| =1$. Thus $D(\lambda) : \C \mapsto \mathcal{B}(\mathcal{H})$ is a differentiable function, and so $\lambda \mapsto T(\lambda)$ is an analytic family of bounded operators on $\C$.
Suppose that $E \in \R \setminus \sigma(T)$. Then
\[ (T(\lambda) -E) = \left(1 + D(\lambda) (T-E)^{-1} \right) (T-E).\]
Thanks to the inequality $|1-e^{z} | \leqslant | z| e^{| z |}$, for all $z \in \C$, we get
\[ \| D(\lambda) \| \leqslant 2d \cdot e^{|\lambda |} | \lambda |. \] 
Also note that $\|(T-E)^{-1}\| \leqslant 1/\text{dist}(\sigma(T),E)$ since $T$ is self-adjoint. Therefore if we require that $|\lambda |$ satisfies \eqref{condition}, it follows that $\| D(\lambda) (T-E)^{-1} \| <1/2$ and we may invert $(T(\lambda) -E)$. Consequently, bounding above by a geometric series gives 
\[ \|(T(\lambda) -E)^{-1} \|  \leqslant  \| (T-E)^{-1} \| \| (1+ D(\lambda) (T-E) ^{-1})^{-1} \|  \leqslant 2 / \text{dist}(\sigma(T),E). \]
\qed
\end{proof}

\noindent \textit{Proof of Theorem \ref{CTmethod}:} Suppose first that $V$ has compact support in $\Z^d$. Then the condition $\mathrm{dist}(\sigma(\Delta), E) > \limsup_{|n| \to + \infty} |V(n)|$ is automatically true since the right side equals zero. Since $H\psi = (\Delta+V)\psi = E\psi$, we write, for $\lambda \in \R$, 
\[ e^{\i \lambda \rho} \psi = -\left( e^{\i \lambda \rho} (\Delta - E)^{-1} e^{-\i \lambda \rho} \right) (e^{\i \lambda \rho} V \psi) = - \left( \Delta (\lambda) -E\right)^{-1} (e^{\i \lambda \rho} V \psi).\]
Because of the analyticity of $\Delta(\lambda)$ and the compactness of the support of $V$, both terms on the right of the previous equation admit an analytic continuation to all of $\C$. Let $\nu$ be the unique positive solution to the equation 
\begin{equation}
\label{condition22}
\R^+ \ni \mu \mapsto \frac{2d \cdot e^{\mu} \mu}{\mathrm{dist}(\sigma(\Delta),E)} =  \frac{1}{2}.
\end{equation}
Set $\lambda = -\i \alpha$, with $\alpha \in (0,\nu)$. Taking norms and applying Proposition \ref{propYouGo} with $T \equiv \Delta$, we see that there exists a constant $C_{E,V,\psi}$ depending on $E$, $V$ and $\psi$, so that 
\[ \| e^{\alpha \rho } \psi \| \leqslant 2 \| \psi \| \cdot  \sup_{n \in \Z^d} | e^{\alpha \rho} V (n)| / \text{dist}(\sigma(\Delta),E)  := C_{E,V,\psi}.\]
We now assume that the support of $V$ is not compact, but $\limsup_{|n| \to +\infty} |V(n)|  < \mathrm{dist}(\sigma(\Delta),E)$ holds. We may write $V = V_c + V_{l}$, where $V_c$ is compactly supported and $\| V_l \| = \sup_{n \in \Z^d} | V_l (n)| \leqslant l$ for some $l < \mathrm{dist}(\sigma(\Delta), E)$. Consider the operator $H_{l}:= \Delta+V_l$. Since $V_l$ is a bounded operator, $H_l(\lambda)$ is an analytic family. If $\epsilon >0$ is any number verifying $\epsilon < \mathrm{dist}(\sigma(\Delta), E) - l $, then $H_l$ has a spectral gap around $E$ of size at least $\epsilon$. This is due to the following spectral inclusion formula, see e.g. \cite[Theorem 3.1]{K1}:
\[ \sigma(H_l) \subset \{ \mu \in \R : \text{dist}\left( \sigma(\Delta),\mu \right) \leqslant \| V_l \| \}. \]
In particular, $(H_l -E)$ is invertible. Since 
\[ (H_l- E) = \left(1+ V_l (\Delta - E)^{-1} \right) (\Delta -E) \]
and  $\| V_l (\Delta-E)^{-1} \| < l / \mathrm{dist}(\sigma(\Delta), E) < 1$, we get
\[ (H_l- E) ^{-1}= (\Delta -E)^{-1} \left(1+ V_l (\Delta - E)^{-1} \right)^{-1} . \]
From the eigenvalue equation $H \psi = (H_l + V_c) \psi = E \psi$, we may write
\[e^{\i \lambda \rho}  \psi = - (H_l (\lambda) - E) ^{-1} (e^{\i \lambda \rho} V_c \psi). \]
Let $\nu$ be the unique positive solution to the equation 
\begin{equation}
\label{condition222}
\R^+ \ni \mu \mapsto \frac{2d \cdot e^{\mu} \mu }{\mathrm{dist}(\sigma(H_l),E)} =  \frac{1}{2}.
\end{equation}
Set $\lambda = -\i \alpha$, with $\alpha \in (0,\nu)$. Taking norms and applying Proposition \ref{propYouGo} with $T \equiv H_l$, we see that there exists a constant $C_{E,V,\psi}$ so that 
\begin{equation*}
 \| e^{\alpha \rho } \psi \| \leqslant 2 \| \psi \| \cdot \sup_{n \in \Z^d} | e^{\alpha \rho} V_c (n)|  / \text{dist}(\sigma(H_l),E) := C_{E,V,\psi}.
 \end{equation*}
\qed






\section{The multidimensional case : sub-exponential decay of eigenfunctions}
\label{MultiD}

We begin this section by fixing more notation, and build on the one introduced above. Let 
\begin{equation*}
\Delta_i := 2-S_i^*-S_i \quad \text{and}
\end{equation*}
\begin{equation*}
A_{0,i} := -\i\left( 2^{-1}(S_i^*+S_i) + N_i(S_i^*-S_i) \right) = \i \left(2^{-1}(S_i^*+S_i) - (S_i^*-S_i)N_i \right).
\end{equation*}
Let
\begin{equation}
A_i' := \i A_{0,i}, \quad \text{and} \quad A' := \sum_{i=1}^d A_i' = i A_0, \quad \text{with} \quad \mathcal{D}(A') = \mathcal{D}(A_0).
\end{equation}
Then the following is a non-negative operator on $\mathcal{H}$:
\begin{equation*}
[\Delta_i, A_i']_{\circ} = \Delta_i (4-\Delta_i)=2-(S_i^*)^2-(S_i)^2.
\end{equation*} 
A useful identity relating the shift operators and the potential is:
\begin{equation}
\label{commuteSVd}
S_i V = (\T_i V) S_i \quad \text{and} \quad S_i^*V = (\T_i^*V)S_i^*.
\end{equation}
Consider an increasing function $F \in C^{3}([0,\infty))$ with bounded derivative away from the origin. Ideally we would like to take $F(x) = \alpha x$ later on, with $\alpha \geqslant 0$ as in \cite{FH}, but it will turn out that slightly better decay conditions on the derivative are required. So examples to keep in mind for a later application are $F_{s,\alpha,\gamma} : [0,\infty) \mapsto [0,\infty)$, where $(s, \alpha, \gamma) \in [0,\infty) \times [0,\infty) \times [0,2/3)$ and
\begin{equation}
\label{FunctionF}
F_{s,\alpha,\gamma}(x) := \Upsilon_s(\alpha x^{\gamma}).
\end{equation}
Here $\Upsilon_s$ is an interpolating function defined for $s\geqslant 0$ by 
\begin{equation}
\label{Upsilon}
\Upsilon_s(x) := \int _0 ^x \langle st\rangle ^{-2}dt.
\end{equation}
Then $\Upsilon_s(x) \uparrow x$ as $s \downarrow 0$, and 
\begin{equation}
\label{espn1gn}
\Upsilon_s(x) \leqslant c_s \quad \text{for} \quad s>0, \quad \text{and} \quad |\Upsilon_s^{(n)}(x)| \leqslant cx^{-n+1},
\end{equation}
where the first constant in \eqref{espn1gn} depends on $s$ whereas the second one does not. It is readily seen that there are constants $C>0$ not depending on $s$ and $\gamma$ such that
\begin{equation}
\label{DerivativesF}
|F'_{s,\alpha,\gamma}(x)| \leqslant C x^{\gamma-1} \quad \text{and} \quad |F''_{s,\alpha,\gamma}(x)| \leqslant C x^{\gamma-2}.
\end{equation}
We also have that for all $x \geqslant 0$,
\begin{equation}
\label{Johny on the spot}
F'_{s,\alpha,\gamma}(x) \geqslant 0 \quad \text{and} \quad F''_{s,\alpha,\gamma}(x) \leqslant 0.
\end{equation}
So $F_{s,\alpha,\gamma}$ is increasing and concave.

For $n=(n_1,...,n_d) \in \Z^d$, let $\langle n \rangle := \sqrt{1+|n|^2}$. The function $F$ induces a radial operator of multiplication on $\mathcal{H}$, also denoted by $F$ and acting as follows: $(Fu)(n) := F(\langle n \rangle)u(n)$, $\forall u \in \mathcal{H}$. For $i=1,...,d$, we introduce the multiplication operators on $\mathcal{H}$:
\begin{equation}
\label{varphiD}
\varphi_{\ell_i} := (\T_i e^F - e^F)/e^F = e^{\T_i F -F} - 1 \quad \text{and} \quad \varphi_{r_i} := (\T^*_i e^F - e^F)/e^F = e^{\T^*_i F -F} - 1,
\end{equation}
\begin{equation}
\label{grD}
g_{\ell_i} := \varphi_{\ell_i}/N_i \quad\text{and} \quad
g_{r_i} := \varphi_{r_i}/N_i.
\end{equation}
In other words, if $U_i : \Z^d \mapsto \Z^d$ denotes the flow $(n_1,...,n_d) \mapsto (n_1,...,n_i-1,...,n_d)$ and $U_i^{-1}$ its inverse, then $\varphi_{\ell_i}$ and $\varphi_{r_i}$ are multiplication at $n$ respectively by $\varphi_{\ell_i}(n)=e^{F(\langle U_i n \rangle) - F(\langle n \rangle)}-1$ and $\varphi_{r_i}(n)=e^{F(\langle U_i^{-1}n \rangle-F(\langle n \rangle)}-1$, while $g_{\ell_i}$ and $g_{r_i}$ are multiplication at $n$ respectively by $g_{\ell_i}(n)=\varphi_{\ell_i}(n)/n_i$ and $g_{r_i}(n)=\varphi_{r_i}(n)/n_i$.
Since $g_{\ell_i}(n)$ and $g_{r_i}(n)$ are not well-defined when $n_i=0$, set $g_{\ell_i}(n) = g_{r_i}(n) := 0$ in that case. We will need the operator $g$ on $\mathcal{H}$ given by 
\begin{equation}
\label{GX}
(gu)(n) = g(n)u(n) := \frac{F'(\langle n \rangle)}{\langle n \rangle} u(n).
\end{equation}
Three remarks are in order. First, by the Mean Value Theorem, $F'$ bounded away from the origin ensures that $\varphi_{\ell_i}$, $\varphi_{r_i}$, $g_{\ell_i}$ and $g_{r_i}$ are bounded operators on $\mathcal{H}$; secondly, $F$ increasing implies $\sign(n_i)\varphi_{r_i}(n) \geqslant 0$, $\sign(n_i)\varphi_{\ell_i}(n) \leqslant 0$, $g_{r_i}(n) \geqslant 0$, $g_{\ell_i}(n) \leqslant 0$ and $g(n) \geqslant 0$; and thirdly, we remark that $F,\varphi_{\ell_i}$, $\varphi_{r_i}$ and $g$ are radial potentials on $\mathcal{H}$.

\begin{proposition}
\label{propdodoD}
Suppose that Hypothesis 1 holds for the potential $V$. Let $F$ be a general function as described above and suppose that for all $i,j=1,...,d$,
\begin{align*}
\cdot \ \ \dagger_1& \quad |g_{r_i}| \in O(1) \quad \text{and} \quad |g_{{\ell}_i}| \in O(1),\\
\cdot \ \ \dagger_2& \quad |\T_i g -g|N_j \in O(1),\\
\cdot \ \ \dagger_3& \quad |\T_i \varphi_{r_i} - \varphi_{r_i}|N_j, \quad |\T_i \varphi_{\ell_i} - \varphi_{\ell_i}|N_j, \quad |\T_i \varphi_{r_j} - \varphi_{r_j}|N_i \quad \text{and} \quad |\T_i \varphi_{\ell_j} - \varphi_{\ell_j}|N_i \in O(1), \\
\cdot \ \ \dagger_4& \quad |(g_{r_i}-g)-(g_{\ell_i}+g)|N_iN_j \in O(1).
\end{align*}
Suppose that $H\psi = E\psi$, with $\psi \in \mathcal{H}$. Let $\psi_F := e^F\psi$, and assume $\psi_F \in \mathcal{H}$. Then $\psi_F \in \mathcal{D}(\sqrt{g}A')$ and there exist bounded operators $(W_i)_{i=1}^d$, $\mathcal{L}$, $\mathcal{M}$ and $\mathcal{G}$ on $\mathcal{H}$ depending on $F$ such that 
\begin{equation}
\begin{aligned}
\label{theZerothoneD}
\big \langle \psi_F, [H,A']_{\circ} \psi_F \big \rangle &=  - 2 \big\|\sqrt{g}A' \psi_F \big\|^2 - \sum_{i=1}^d \big\|\sqrt{\Delta_i(4-\Delta_i)} W_i \psi_F \big\|^2 \\
& \quad + 2^{-1} \big \langle \psi_F, (\mathcal{L} + \mathcal{M} + \mathcal{G}) \psi_F \big \rangle. 
\end{aligned}
\end{equation}
The $W_i$ are multiplication operators given by $W_i = W_{F;i} := \sqrt{\cosh(\T_i F-F)-1}$. The expressions of $\mathcal{L}$, $\mathcal{M}$ and $\mathcal{G}$ are involved; they are given by \eqref{LL}, \eqref{MM} and \eqref{GG} respectively. The relevant point is that these three operators are a finite sum of terms, each one of the form
\begin{equation}
\label{Hotbabe}
    P_1(S_1,...,S_d,S_1^*,...,S_d^*) T P_2(S_1,...,S_d,S_1^*,...,S_d^*), 
\end{equation}
where $P_1$ and $P_2$ are multivariable polynomials in $S_1,...,S_d,S_1^*,...,S_d^*$ and $T$ are multiplication operators of the kind listed in $\dagger_1 - \dagger_4$.
\end{proposition}
\begin{remark}
Formula \eqref{theZerothoneD} has an additional negative term compared to the corresponding formula for the continuous Schr\"{o}dinger operator, cf.\ \cite[Lemma 2.2]{FH}: 
\begin{equation*}
\langle \psi_F, [H,A']_{\circ}\psi_F \rangle = -4 \|\sqrt{g} A' \psi_F \|^2 + \langle \psi_F, \mathcal{Q} \psi_F \rangle, \quad \text{with} \quad \mathcal{Q} = (x \cdot \nabla)^2g -x\cdot\nabla(\nabla F)^2.     
\end{equation*}
\end{remark}
\begin{remark}
As mentioned in \cite{FH}, if we consider the Virial Theorem disregarding operator domains, it is reasonable to expect $\langle \psi, [H,e^F A' e^F]  \psi \rangle = 0$. This idea underlies \eqref{theZerothoneD}.
\end{remark}
\begin{proof}
Let $\phi \in \ell_0(\Z^d)$, the sequences with compact support, and $\phi_F := e^F\phi$. The first step of the proof consists in establishing the following identity :
\begin{equation}
\begin{aligned}
\label{theFirstoneD}
\big \langle \phi, [e^F A' e^F,\Delta] \phi \big \rangle &= \big \langle \phi_F, [A', \Delta] \phi_F \big \rangle - 2\big\|\sqrt{g}A' \phi_F \big\|^2  \\
&\quad - \sum_{1 \leqslant i \leqslant d} \big\|\sqrt{\Delta_i (4-\Delta_i)} W_i \phi_F \big\|^2 + 2^{-1} \big \langle \phi_F, (\mathcal{L} + \mathcal{M} + \mathcal{G}) \phi_F \big \rangle.
\end{aligned}
\end{equation}
The proof of \eqref{theFirstoneD} is technical and long, so it is done in the Appendix. The assumptions of this Proposition together with $F'$ bounded away from the origin imply that the $W_i$, $\mathcal{L}$, $\mathcal{M}$ and $\mathcal{G}$ stemming from this calculation are bounded operators. Exactly where these assumptions are applied are indicated in the Appendix by $(\ddagger)$. The second step consists in using \eqref{theFirstoneD} to prove \eqref{theZerothoneD}. For $m \geqslant 1$, define the cut-off potentials $\chi_m(n) := \chi(\langle n \rangle /m)$ on $\Z^d$, where $\chi \in C^{\infty}_c(\R)$ and $\chi$ equals one in a neighborhood of the origin. 
Then \eqref{theFirstoneD} holds with $\phi = \chi_m \psi$ and $\phi_F = e^F\chi_m \psi$. Adding $\big \langle \chi_m \psi, [e^F A' e^F, V] \chi_m \psi \big \rangle = \big \langle e^F \chi_m \psi, [A', V] e^F \chi_m \psi \big \rangle $ to each side of \eqref{theFirstoneD}, and introducing the constant $E$ in the commutator on the left gives
\begin{equation}
\begin{aligned}
\label{theSecondoneD}
\big \langle \chi_m \psi, [e^F A' e^F,H-E] \chi_m \psi \big \rangle &= \big \langle e^F \chi_m \psi, [A', H] e^F  \chi_m \psi \big \rangle - 2 \big\|\sqrt{g}A' e^F \chi_m \psi \big\|^2 \\
& \quad - \sum_{1 \leqslant i \leqslant d} \big\|\sqrt{\Delta_i(4-\Delta_i)} W_i e^F \chi_m \psi \big\|^2 \\
&\quad + 2^{-1}\big \langle e^F \chi_m \psi, (\mathcal{L} + \mathcal{M} + \mathcal{G}) e^F \chi_m \psi \big \rangle. 
\end{aligned}
\end{equation}
Since $e^F\chi_m \psi \to \psi_F$ in $\mathcal{H}$ as $m \to \infty$, the first, third and fourth terms on the right side of \eqref{theSecondoneD} converge. The left side of \eqref{theSecondoneD} is handled in the same way as in \cite[Proposition 4.16]{CFKS}:
\begin{align*}
\big \langle \chi_m \psi, [e^F A' e^F,H-E] \chi_m \psi \big \rangle &= - 2\Re \big(\big \langle e^F A' e^F \chi_m \psi, (H-E) \chi_m \psi \big \rangle \big) \\
&=- 2\Re \big(\big \langle \langle N \rangle^{-1} A' e^F \chi_m \psi, \langle N \rangle e^F (H-E) \chi_m \psi \big \rangle \big).
\end{align*}
Since $\supp(\chi_m) \subset [-2m,2m]^d$, $\supp((H-E)\chi_m\psi) \subset K := [-2m-1,2m+1]^d$ and so commuting $\chi_m$ with $(H-E)$ gives
\begin{align}
\langle N \rangle e^F (H-E) \chi_m \psi &= \langle N \rangle e^F \mathbf{1}_K (H-E) \chi_m \psi \nonumber \\
& = \sum_{1 \leqslant i \leqslant d} \langle N \rangle(\chi_m - \T_i\chi_m) e^F S_i\psi+\langle N \rangle(\chi_m -\T_i^* \chi_m) e^F S_i^*\psi.\label{ddidD}
\end{align}
An application of the Mean Value Theorem shows that $|\langle N \rangle (\chi_m - \T_i\chi_m)|$ and $|\langle N \rangle (\chi_m - \T_i^* \chi_m)|$ are bounded by a constant independent of $m$. Moreover, $\psi_F \in \mathcal{H}$ and $F'$ bounded imply that $e^F S_i \psi = S_i e^{\T_i^* F -F}\psi_F$ and $e^F S_i^* \psi= S_i^*e^{\T_i F-F} \psi_F \in \mathcal{H}$. Thus the sequence \eqref{ddidD} is uniformly bounded in absolute value in $\mathcal{H}$. Furthermore, it converges pointwise to zero. By Lebesgue's Dominated Convergence Theorem, 
\begin{equation}
\big\|\langle N \rangle e^F (H-E) \chi_m \psi\big \| \to 0 \quad \text{as} \quad m\to \infty.
\end{equation}
Since $\langle N \rangle^{-1} A'$ is a bounded operator on $\mathcal{H}$, the left side of \eqref{theSecondoneD} converges to zero as $m \to \infty$. The only remaining term in \eqref{theSecondoneD} is 2$\|\sqrt{g} A' e^F \chi_m \psi\|^2$, hence it must also converge as $m \to \infty$. To finish the proof, it remains to show that $\psi_F \in \mathcal{D}(\sqrt{g} A')$. Let $\phi \in \ell_0(\Z)$. Then
\begin{equation*}
\big| \big \langle \psi_F, A' \sqrt{g}\phi \big \rangle \big | = \lim \limits_{m\to \infty} \big| \big \langle e^F \chi_m \psi, A' \sqrt{g}\phi \big \rangle \big | \leqslant \Big( \lim \limits_{m\to \infty} \| \sqrt{g} A' e^F \chi_m \psi \| \Big) \|\phi\|.
\end{equation*}
This shows that $\psi_F \in \mathcal{D}\big((-A'\sqrt{g})^*\big) = \mathcal{D}(\sqrt{g} A')$. Then it must be that $\|\sqrt{g} A' e^F \chi_m \psi\|^2 \to \|\sqrt{g} A' \psi_F\|^2$ and the proof is complete after rearranging the terms accordingly in \eqref{theSecondoneD}.
\qed
\end{proof}

As mentionned in the last Proposition, $\mathcal{L},\mathcal{M}$ and $\mathcal{G}$ are a finite sum of terms of the form
\begin{equation*}
P_1(S_1,...,S_d,S_1^*,...,S_d^*) T P_2(S_1,...,S_d,S_1^*,...,S_d^*) 
\end{equation*}
for some polynomials $P_1$ and $P_2$. Going forward, it is essential that the multiplication operators $T=T(n)$ decay radially at infinity. In other words, for the minimal assumptions $\dagger_1-\dagger_4$, we will need $o(1)$ instead of $O(1)$. The following Lemma shows that this is the case for $F = F_{s,\alpha,\gamma}$.

\begin{Lemma}
\label{FunGuyWitch}
Let $F = F_{s, \alpha,\gamma}$ be the function defined in \eqref{FunctionF}. Consider its corresponding functions $\varphi_{r_i}, \varphi_{\ell_i}, g_{r_i}, g_{\ell_i}$ and $g$. The following estimates hold uniformly with respect to $s$ and $\gamma$:
\begin{align*}
\cdot \ \ \ddagger_1& \quad  |g_{r_i}| \quad \text{and} \quad  |g_{\ell_i}| \in O_{\alpha}( \langle n \rangle ^{\gamma-2}), \\
\cdot \ \ \ddagger_2& \quad |\T_i g - g| \in O_{\alpha}(\langle n \rangle ^{\gamma-3}), \\
\cdot \ \ \ddagger_{3}& \quad |\T_i^* \varphi_{r_j} - \varphi_{r_j}| \quad \text{and} \quad  
|\T_i^* \varphi_{\ell_j} - \varphi_{\ell_j}| \in O_{\alpha}(\langle n \rangle ^{\gamma-2}), \\
\cdot \ \ \ddagger_4& \quad |(g_{r_i}-g)-(g_{\ell_i}+g)|  \in O_{\alpha} (\langle n \rangle ^{3\gamma-4}), \\ 
\cdot \ \ \ddagger_5& \quad |(\T_i F -F) - \T_i(\T_i F-F)|  \in O_{\alpha}( \langle n \rangle ^{\gamma-2}).
\end{align*}
Therefore $\ddagger_i$ improve $\dagger_i$ for $i=1,2,3,4$ respectively.
\end{Lemma}
\begin{proof}
These estimates are simple applications of the Mean Value Theorem (MVT). Let $n=(n_1,...,n_d) \in \Z^d$ and fix $i \in \{1,...,d\}$. There is $n'=(n'_1,...,n'_d)$ with $n'_i \in (n_i, n_i+1)$ and $n'_j = n_j$ for $j \neq i$ such that
\begin{equation*}
    g_{r_i}(n) = \frac{n'_i}{\langle n' \rangle} \frac{F'(\langle n' \rangle ) e^{F(\langle n' \rangle )}}{n_i e^{F(\langle n \rangle )}}.
\end{equation*}
This, together with \eqref{DerivativesF}, and an analogous calculation for $g_{\ell_i}(n)$ shows $\ddagger_1$. Define $g : \R^d \to \R$, $g(x) := F'(\langle x\rangle)\langle x \rangle ^{-1}$. Then $\ddagger_2$ follows from 
\begin{equation*}
    \frac{\partial g}{\partial x_i}(x) = \frac{x_i}{\langle x \rangle} \frac{F''(\langle x \rangle) \langle x \rangle - F'(\langle x \rangle)}{\langle x \rangle^2}.
\end{equation*}
Now fix $i,j \in \{1,...,d\}$. First there is $n'=(n_1',...,n_d')$ with $n_j' \in (n_j,n_j+1)$ and $n_k' = n_k$ for $k \neq j$ such that
\begin{equation*}
(\T_j^* F - F)(n) = \frac{\partial \tilde{F}}{\partial x_j} (n') = \frac{n_j'}{\langle n' \rangle}F'(\langle n'\rangle), \quad \text{with} \quad \tilde{F}(x) = F(\langle x \rangle).
\end{equation*}
Then there is $n''=(n_1'',...,n_d'')$ with $n_i'' \in (n_i',n_i'+1)$ and $n_k'' = n_k'$ for $k \neq i$ such that 
\begin{equation*}
(\T_i^* \varphi_{r_j} - \varphi_{r_j})(n) = \frac{\partial^2 \tilde{F}}{\partial x_i \partial x_j} (n'') e ^{\frac{\partial \tilde{F}}{\partial x_j}(n'')}.
\end{equation*}
This proves $\ddagger_3$ since
\begin{equation*}
\left|\frac{\partial^2 \tilde{F}}{\partial x_i \partial x_j} (x) \right| \leqslant \frac{|F'(\langle x \rangle)|}{\langle x \rangle} + |F''(\langle x \rangle)|.
\end{equation*}
The latter estimate on $\partial^2 \tilde{F} /(\partial x_i \partial x_j)$ also implies $\ddagger_5$. Finally, for $\ddagger_4$, we start with
\begin{equation*}
    g_{r_i}(n)-g(n) = \frac{1}{n_i e^{F(\langle n \rangle)}} \bigg[ \frac{n'_i}{\langle n' \rangle} F'(\langle n' \rangle ) e^{F(\langle n' \rangle )} - \frac{n_i}{\langle n \rangle} F'(\langle n \rangle ) e^{F(\langle n \rangle )} \bigg] = \frac{1}{n_i e^{F(\langle n \rangle)}}\frac{\partial k}{\partial x_i}(n'')
\end{equation*}
where
\begin{equation*}
k: \R^d \to \R, \quad k(x) := \frac{x_i}{\langle x \rangle}F'(\langle x \rangle) e^{F(\langle x \rangle)},
\end{equation*}
and $n''=(n_1'',...,n''_d)$ with $n_i'' \in (n_i,n_i')$ and $n_j'' = n_j$ for $j\neq i$. We compute 
\begin{equation*}
\frac{\partial k}{\partial x_i}(x) = \left(\frac{F'(\langle x \rangle)}{\langle x \rangle} - \frac{x_i^2 F'(\langle x \rangle)}{\langle x \rangle^3} + \frac{x_i^2 F''(\langle x \rangle)}{\langle x \rangle^2} + \frac{x_i^2 (F'(\langle x\rangle))^2}{\langle x \rangle ^2} \right) e^{F(\langle x \rangle)}.
\end{equation*}
Thus for some $n'''=(n_1''',...,n_d''')$ with $n_i''' \in (n_i-1,n_i+1)$ and $n_j'''=n_j$ for $j\neq i$, we have
\begin{equation*}
(g_{r_i}(n)-g(n))-(g_{\ell_i}(n)+g(n)) = \frac{1}{n_i e^{F(\langle n \rangle}}\frac{\partial^2 k}{\partial x_i^2}(n''').    
\end{equation*}
A calculation of $\partial^2 k / \partial x_i^2$ yields the required estimate.
\qed
\end{proof}

We are now ready to prove the main result concerning the multi-dimensional operator $H$: 
\vspace{0.5cm}

\noindent \textit{Proof of Theorem \ref{sunnyHunny}}. \quad Let $\psi_{F_{s,\alpha,\gamma}} := e^{F_{s,\alpha,\gamma}}\psi$, and let $\Psi_{s} := \psi_{F_{s,\alpha,\gamma}}/\|\psi_{F_{s,\alpha,\gamma}}\|$. We suppose that for some $(\alpha,\gamma) \in [0,\infty) \times [0,2/3)$, $\psi \not \in \mathcal{D}(\vartheta_{\alpha,\gamma})$ and derive a contradiction. Of course, $\psi_{F_{s,\alpha,\gamma}} \in \mathcal{H}$ for all $s>0$, but by the Monotone Convergence Theorem, $\|\psi_{F_{s,\alpha,\gamma}} \| \to + \infty$ as $s \downarrow 0$. Thus, for any bounded set $B \subset \Z^d$, 
\begin{equation}
\label{weaklyBud}
    \lim \limits_{s \downarrow 0} \sum_{n \in B} |\Psi_{s}(n)|^2 = 0.
\end{equation}
In particular, $\Psi_s$ converges weakly to zero. As $\alpha$ and $\gamma$ are fixed, we shall write $F_s$ instead of $F_{s,\alpha,\gamma}$ for simplicity. Introduce the operator $H_{F_s} := e^{F_s} H e^{-F_s}$. Then $H_{F_s}$ is a bounded operator and $H_{F_s} \Psi_{s} = E \Psi_{s}$. We claim that 
\begin{equation}
\label{Jimmy}
    \lim \limits_{s \downarrow 0} \|(H-E)\Psi_{s}\| = 0.
\end{equation}
To see this, write $H_{F_s}$ as follows:
\begin{equation*}
    H_{F_s} = H + \sum_{1\leqslant i \leqslant d} S_i (1-e^{\T_i^*F_s-F_s}) + S_i^*(1-e^{\T_i F_s -F_s}).
\end{equation*}
To show \eqref{Jimmy}, it is therefore enough to show that 
\begin{equation}
\label{bunHunny}
\lim \limits_{s \downarrow 0} \|(1-e^{\T_i^* F_s-F_s})\Psi_s\| = \lim \limits_{s \downarrow 0} \|(1-e^{\T_i F_s-F_s})\Psi_s\| =0.
\end{equation} 
Let $B(N) = \{n \in \Z^d : \langle n \rangle  \leqslant N\}$, and $B(N)^{\text{c}}$ the complement set. For all $\epsilon >0$, there is $N >0$ such that
\begin{equation*}
    \sup \limits_{\substack{n \in B(N)^{\text{c}} \\ s > 0}} \left|1-e^{(\T_i^* F_s-F_s)(n)}\right| = \sup \limits_{\substack{n \in B(N)^{\text{c}} \\ s > 0}} \left|1-e^{\alpha\gamma \langle n'\rangle^{\gamma-1} \Upsilon'_s(\alpha \langle n' \rangle ^{\gamma})}\right| \leqslant \epsilon 
\end{equation*}
(here $n'=(n_1',...,n_d')$ with $n_i' \in (n_i,n_i+1)$ and $n_j' = n_j$ for $j\neq i$). Combining this with \eqref{weaklyBud} proves the first limit in \eqref{bunHunny}, and the second one is shown in the same way. Thus the claim is proven. Because $E \in \Theta(H)$, there exists an interval $\Sigma := (E-\delta,E+\delta)$ with $\delta >0$, $\eta >0$ and a compact $K$ such that 
\begin{equation}
\label{GunShy}
    E_{\Sigma}(H) [H,A']_{\circ} E_{\Sigma}(H) \geqslant \eta E_{\Sigma}(H) + K.
\end{equation}
By functional calculus, 
\begin{equation}
\label{gitter}
\lim \limits_{s \downarrow 0} \|E_{\R \setminus \Sigma}(H) \Psi_s\| \leqslant \lim \limits_{s \downarrow 0} \delta^{-1} \|E_{\R \setminus \Sigma}(H)(H-E)\Psi_s\| =0.
\end{equation}
It follows by the Mourre estimate \eqref{GunShy} and \eqref{gitter} that 
\begin{equation}
\label{SunnyT}
\liminf \limits_{s \downarrow 0} \ \langle \Psi_s, E_{\Sigma}(H) [H,A']_{\circ} E_{\Sigma}(H) \Psi_s \rangle \geqslant \eta \liminf \limits_{s\downarrow 0}  \|E_{\Sigma}(H) \Psi_s \|^2 = \eta > 0.
\end{equation}
We now look to contradict this equation. We start with
\begin{equation}
\label{TryFoG}
\langle \Psi_s, E_{\Sigma}(H) [H,A']_{\circ} E_{\Sigma}(H) \Psi_s \rangle = \langle \Psi_s,  [H,A']_{\circ} \Psi_s \rangle - f_1(s) - f_2(s), \quad \text{where}
\end{equation}
\begin{equation*}
f_1(s) = \langle \Psi_s, E_{\R \setminus \Sigma}(H) [H,A']_{\circ} E_{\Sigma}(H) \Psi_s \rangle \quad \text{and} \quad
f_2(s) = \langle \Psi_s,  [H,A']_{\circ} E_{\R \setminus \Sigma}(H) \Psi_s \rangle.
\end{equation*}
Applying \eqref{gitter} gives
\begin{equation*}
    \lim \limits_{s \downarrow 0} |f_1(s) | = \lim \limits_{s \downarrow 0} |f_2(s) | =0.
\end{equation*}
Now apply \eqref{theZerothoneD} with $F = F_{s,\alpha,\gamma}$, and after dividing this equation by $\|\Psi_s\|^2$, we have 
\begin{equation*}
    \limsup \limits_{s \downarrow 0} \ \langle \Psi_s,  [H,A']_{\circ} \Psi_s \rangle \leqslant 0.
\end{equation*}
Here we took advantage of the negativity of the first two terms on the right side of \eqref{theZerothoneD}, and used the uniform decay of $\mathcal{L}+\mathcal{M}+\mathcal{G}$ together with the weak convergence of $\Psi_s$ to get $\langle \Psi_s, (\mathcal{L}+\mathcal{M}+\mathcal{G}) \Psi_s \rangle \to 0$ as $s \downarrow 0$. To check this thoroughly, one needs to apply the estimates of Lemma \ref{FunGuyWitch} to where indicated in the Appendix by a $(\ddagger)$. Note that $\mathcal{L}$ given by \eqref{LL} is the most constraining term; it has the necessary decay provided $3\gamma-4 <-2$, i.e.\ $\gamma < 2/3$. Note also that $\ddagger_5$ allows to conclude, by continuity of the map $x\mapsto \sqrt{\cosh(x)-1}$, that $\langle \Psi_s, (W_{F_s;i}-\T_i W_{F_s;i})\Psi_s \rangle$ and like terms converge to zero. 
Thus by \eqref{TryFoG},
\begin{equation*}
\limsup \limits_{s \downarrow 0} \ \langle \Psi_s, E_{\Sigma}(H) [H,A']_{\circ} E_{\Sigma}(H) \Psi_s \rangle \leqslant 0.
\end{equation*}
This is in contradiction with \eqref{SunnyT}, so the proof is complete.
\qed




\section{Proof of Proposition \ref{WVN}}
\label{SectionWVN}
As an application of Theorem \ref{sunnyHunny}, we display a Wigner-von Neumann type operator that has an eigenvalue embedded in the essential spectrum. The eigenvalue is proven to be a threshold.

\noindent \textit{Proof of Proposition \ref{WVN}.}
First, we construct the potential in dimension one. Second, we generalize this potential to higher dimensions. Third, we show that the eigenvalue is also a threshold and belongs to the essential spectrum.

\noindent \textbf{Part 1.}  We follow \cite[Section XIII.13, Example 1]{RS4}. Starting with the eigenvalue equation
\[ 2\psi(n)-\psi(n+1)-\psi(n-1) + V(n) \psi(n) = E\psi(n),\] 
we shift terms to write 
\[V(n) = (E-2) + \frac{\psi(n+1)}{\psi(n)} + \frac{\psi(n-1)}{\psi(n)}.\]
We try the Ansatz $\psi(n) := \sin(kn) w_k(n)$, $k \in (0,\pi)$. For simplicity, write $w(n)$ instead of $w_k(n)$. We get 
\begin{align*} 
V(n) &= (E-2) \\
& \quad + \frac{\sin(kn)\cos(k)+\cos(kn)\sin(k)}{\sin(kn)}\frac{w(n+1)}{w(n)} + \frac{\sin(kn)\cos(k)-\cos(kn)\sin(k)}{\sin(kn)}\frac{w(n-1)}{w(n)} \\
&=(E-2) + \cos(k)\left(\frac{w(n+1)}{w(n)}+\frac{w(n-1)}{w(n)}\right)+\sin(k)\frac{\cos(kn)}{\sin(kn)}\left(\frac{w(n+1)}{w(n)}-\frac{w(n-1)}{w(n)}\right).
\end{align*}
For the moment, let us assume that 
\begin{equation}
\label{firsty}
\frac{w(n+1)}{w(n)} \to 1, \quad \text{as} \ |n| \to + \infty
\end{equation}
and
\begin{equation}
\label{thirsty}
\sin(k)\frac{\cos(kn)}{\sin(kn)}\left(\frac{w(n+1)}{w(n)}-\frac{w(n-1)}{w(n)}\right) \to 0, \quad \text{as} \ |n| \to + \infty.
\end{equation}
Thus if we want $V(n) \to 0$, we must have $(E-2) + 2\cos(k) = 0$, i.e.\ $E = 2-2\cos(k)$.
We now seek a suitable $w_k$. 
Let 
\[g_k(n) = g(n) := \sin(2k) n -\sin(2kn).\] 
For simplicity, we would like to define $w_k(n) := 1/g_k(n)$.
But then $w_k(-1)$, $w_k(0)$ and $w_k(1)$ are not well-defined, nor is $w_k$ for that matter if $k=\pi/2$. To circumvent this problem, we could define $w_k(n) := (1+(g_k(n))^2)^{-1}$ instead, as it is done in \cite[Section XIII.13, Example 1]{RS4}, but alternatively we note that there is $t = t_k \in (0,+\infty)$ such that $t_k +g_k(n) = 0$ has no solutions for $n \in \Z$. So we let 
\[ w_k(n) := \frac{1}{t_k + g_k(n)}.\]
In any case, with either choice we certainly have $\psi \in \ell^2(\Z)$ and \eqref{firsty} is clearly satisfied. As for \eqref{thirsty}, we calculate
\begin{align*}
\sin(k)\frac{\cos(kn)}{\sin(kn)} \left(\frac{w(n+1)}{w(n)}-\frac{w(n-1)}{w(n)}\right) &= \sin(k)\frac{\cos(kn)}{\sin(kn)} \frac{g(n-1)-g(n+1)}{[t+g(n-1)][t+g(n+1)]} [t+g(n)] \\
&= \frac{-2\sin(k)\sin(2k)\sin(2kn)}{[t+g(n-1)][t+g(n+1)]} [t+g(n)] \\
& = \frac{-2\sin(k)\sin(2kn)}{n} + O(n^{-2}).
\end{align*}
So \eqref{thirsty} also holds. Note that this calculation follows from these useful relations: 
\[g(n+1)-g(n) = \sin(2k) - 2\sin(k)\cos(2kn+k),\]
\[\frac{1}{[t+g(n+1)]} = \frac{1}{\sin(2k) n} + O(n^{-2}), \quad \text{and} \quad \frac{1}{[t+g(n-1)]} = \frac{1}{\sin(2k) n} + O(n^{-2}).\]
Letting $E = 2-2\cos(k)$, we then find that $V$ is given by 
\begin{align*} 
V(n) & = \cos(k) \left( \frac{2t + g(n-1)+g(n+1)}{[t+g(n-1)][t+g(n+1)]}[t+g(n)] -2\right) - \frac{2\sin(k)\sin(2k)\sin(2kn)[t+g(n)]}{[t+g(n-1)][t+g(n+1)]} \\
&= \cos(k) \left( \frac{g(n)-g(n-1)}{[t+g(n-1)]}-\frac{g(n+1)-g(n)}{[t+g(n+1)]} \right) - \frac{2\sin(k)\sin(2k)\sin(2kn)[t+g(n)]}{[t+g(n-1)][t+g(n+1)]}.
\end{align*}
By a calculation done above, we know the asymptotic behavior of the second term of this expression. Another calculation shows that the first term of this expression has the exact same asymptotic behavior as the second. 
Thus, we have found a potential having the property that $2-2\cos(k)$ is an eigenvalue of $\Delta +V$ with eigenvector given by $\psi(n) = \sin(kn) [t_k + \sin(2k)n-\sin(2kn)]^{-1}$. Moreover the potential has the asymptotic behavior 
\[ V(n) = -\frac{4\sin(k)\sin(2kn)}{n} + O_{k,t_k}(n^{-2}).\]
\noindent \textbf{Part 2.} We simply extend to two dimensions. The Schr\"odinger equation is rewritten as follows:
\[V(n,m) = (E-4) + \frac{\psi(n+1,m)}{\psi(n,m)} + \frac{\psi(n-1,m)}{\psi(n,m)} + \frac{\psi(n,m+1)}{\psi(n,m)} + \frac{\psi(n,m-1)}{\psi(n,m)}.\]
Try the Ansatz $\psi(n,m) = \sin(k_1 n) w_{k_1}(n) \sin(k_2 m) w_{k_2}(m)$, for some $k_1,k_2 \in (0,\pi)$. For simplicity, write $w_1(n)$ instead of $w_{k_1}(n)$, and $w_2(m)$ instead of $w_{k_2}(m)$. We get 
\begin{align*} 
V(n,m) &= (E-4) \\
& \quad + \cos(k_1)\left(\frac{w_1(n+1)+w_1(n-1)}{w_1(n)}\right) + \sin(k_1)\frac{\cos(k_1n)}{\sin(k_1n)}\left(\frac{w_1(n+1)-w_1(n-1)}{w_1(n)}\right)  \\
& \quad + \cos(k_2)\left(\frac{w_2(m+1)+w_2(m-1)}{w_2(m)}\right) +\sin(k_2)\frac{\cos(k_2m)}{\sin(k_2m)}\left(\frac{w_2(m+1)-w_2(m-1)}{w_2(m)}\right).
\end{align*}
Let $E := 4 - 2\cos(k_1) - 2\cos(k_2)$, and
\[w_1(n) := (t_1 + g_{1}(n))^{-1},  \quad \text{where} \ g_{1}(n) := \sin(2k_1) n -\sin(2k_1n),\]
\[w_2(m) := (t_2 + g_{2}(m))^{-1},  \quad \text{where} \ g_{2}(m) := \sin(2k_2) m -\sin(2k_2m).\]
Here $t_1 = t_{k_1}$ and $t_2 = t_{k_2}$ are real numbers chosen so that $t_1 + g_1(n) \neq 0$ and $t_2 + g_2(m) \neq 0$ for all $n,m \in \Z$. The calculations of the first part show that $V$ is given by
\begin{align*} 
&V(n,m) = \\
& \cos(k_1) \bigg [ \frac{g_1(n)-g_1(n-1)}{t_1+g_1(n-1)}-\frac{g_1(n+1)-g_1(n)}{t_1+g_1(n+1)} \bigg ] - \frac{2\sin(k_1)\sin(2k_1)\sin(2k_1n)[t_1+g_1(n)]}{[t_1+g_1(n-1)][t_1+g_1(n+1)]} \\
& + \cos(k_2) \bigg [ \frac{g_2(m)-g_2(m-1)}{t_2+g_2(m-1)}-\frac{g_2(m+1)-g_2(m)}{t_2+g_2(m+1)} \bigg ] - \frac{2\sin(k_2)\sin(2k_2)\sin(2k_2m)[t_2+g_2(m)]}{[t_2+g_2(m-1)][t_2+g_2(m+1)]}.
\end{align*}
This potential has the property that $4-2\cos(k_1)-2\cos(k_2)$ is an eigenvalue of $\Delta +V$ with eigenvector 
\[ \psi(n,m) = \sin(k_1n) \sin(k_2m) [t_{k_1} + \sin(2k_1)n-\sin(2k_1n)]^{-1}[t_{k_2} + \sin(2k_2)m-\sin(2k_2m)]^{-1}.\]
Moreover $V$ has the asymptotic behavior 
\[ V(n,m) = -\frac{4\sin(k_1)\sin(2k_1n)}{n} -\frac{4\sin(k_2)\sin(2k_2m)}{m} + O_{k_1,t_{k_1}}(n^{-2}) + O_{k_2,t_{k_2}}(m^{-2}).\]
\textbf{Part 3.} We still have to prove that the eigenvalue $E := 4 - 2\cos(k_1) - 2\cos(k_2)$ is a threshold of $H = \Delta +V$. But $V$ satisfies Hypothesis 1, and the eigenvector $\psi$ has slow decay at infinity. So we conclude by Theorem \ref{sunnyHunny} that this eigenvalue is unmistakably a threshold. If $H_1(k)$ denotes the one-dimensional Schr\"odinger operator of Part 1 and $H$ denotes the two-dimensional operator of Part 2, then we have $H = H_1(k_1) \otimes \mathbf{1} + \mathbf{1} \otimes H_1 (k_2)$. A basic result on the spectra of tensor products gives 
\[\sigma(H) = \overline{\sigma(H_1(k_1)) + \sigma(H_1(k_2))} \supset [0,8].\] 
Thus $E \in [0,8] \subset \sigma_{\mathrm{ess}}(H)$. 
\qed















\section{The one-dimensional case: exponential decay of eigenfunctions}
\label{hut89}
In this section we deal with the one-dimensional Schr\"odinger operator $H$ on $\mathcal{H} = \ell^2(\Z)$. We follow the same definitions as in the Introduction and Section \ref{MultiD}, but since $i=1$, we will drop this subscript. We shall write $S$ and $S^*$ instead of $S_i$ and $S_i^*$, $N$ instead of $N_i$, etc...Consider an increasing function $F \in C^2([0,\infty))$ with bounded derivative away from the origin. This function induces a radial operator on $\mathcal{H}$ as in Section \ref{MultiD}: $(Fu)(n) := F(\langle n \rangle)u(n)$ for all $u \in \mathcal{H}$.

\begin{proposition}
\label{propdodo}
Suppose that Hypothesis 1 holds for the potential $V$.
Let $F$ be as above, and suppose additionally that 
\begin{equation}
\label{assF1}
|xF''(x)| \leqslant C, \quad \text{for} \ x \ \text{away from the origin}.
\end{equation}
Suppose that $H\psi = E\psi$, with $\psi \in \mathcal{H}$. Let $\psi_F := e^F\psi$, and assume that $\psi_F \in \mathcal{H}$. Then $\psi_F \in \mathcal{D}(\sqrt{g_r-g_{\ell}}A')$ and there exist bounded operators $W$, $M$ and $G$ depending on $F$ such that
\begin{equation}
\label{theZerothone}
\big \langle \psi_F, [H,A'] \psi_F \big \rangle =  - \big\|\sqrt{g_r-g_{\ell}}A' \psi_F \big\|^2 - \big\|\sqrt{\Delta(4-\Delta)} W \psi_F \big\|^2 + 2^{-1}\big \langle \psi_F,  (M+G) \psi_F \big \rangle. 
\end{equation}
The exact expressions of $W,M$ and $G$ are given by \eqref{theOmegadd}, \eqref{FunnyGuy} and \eqref{goHunt} respectively.
\end{proposition}
\begin{proof}
The proof is done in two steps. The first step consists in proving that 
\begin{equation}
\begin{aligned}
\label{theFirstoned1}
\langle \phi, [e^F A'e^F,\Delta]\phi \rangle &=  \langle \phi_F, [A',\Delta] \phi_F \rangle - \|\sqrt{g_r-g_{\ell}}A' \phi_F\|^2 \\
& \quad - \|\sqrt{\Delta(4-\Delta)} W \phi_F\|^2 + 2^{-1} \langle \phi_F,  (M+G) \phi_F \rangle. 
\end{aligned}
\end{equation}
The proof of this is in the Appendix starting from \eqref{theFirstone}. That $F'$ is bounded away from the origin ensures that $W$ and $(g_r-g_{\ell})$ are bounded. The additional assumption \eqref{assF1} ensures that $(\T^* \varphi_r -\varphi_r)N$ and like terms are bounded. The second step is the same as that of Proposition \ref{propdodoD}, and the proof is identical.
\qed
\end{proof}







 

\begin{Lemma}
\label{Goodlemma}
Suppose that $H\psi = E \psi$ with $\psi \in \ell^2(\Z)$. Let $F$ be a general function as above, and assume that $\psi_F := e^F \psi \in \ell^2(\Z)$. Define the operator
\begin{equation}
\label{H(F)}
H_F := e^{F} H e^{-F}.
\end{equation} 
Then $H_F$ is bounded, $H_F \psi_F = E \psi_F$ and there exist bounded operators $C_F$ and $R_F$ such that
\begin{equation}
\label{gogog}
H_F = C_F H + (2 - 2C_F) + 2^{-1} R_F, \quad \text{where}
\end{equation}
\begin{equation}
\label{CF}
C_F:=2^{-1}\left(e^{F-\T F} + e^{F-\T^* F}\right) \quad \text{and}
\end{equation}
\begin{equation}
\begin{aligned}
\label{dosw}
R_F &:= V(2-2C_F) + (\T \varphi_r -\varphi_r)(S^*-S) + (\varphi_{\ell}-\T^*\varphi_{\ell})(S^*-S) \\
&\quad + (g_r-g_{\ell})A' - 2^{-1}(g_r-g_{\ell})(S^*+S).
\end{aligned}
\end{equation}
\end{Lemma}
\begin{proof}
Because $F'$ is bounded away from the origin, both $e^FSe^{-F}\phi = Se^{\T^*F-F}\phi $ and $e^FS^*e^{-F}\phi = S^*e^{\T F-F}\phi$ belong to $\ell^2(\Z)$ whenever $\phi \in \ell^2(\Z)$. Thus $H_F$ is bounded, and $H_F \psi_F = E \psi_F$ follows immediately. Now
\begin{equation*}
\label{siri}
H_F = 2+V-e^{F-\T F}S - e^{F-\T^* F}S^*.
\end{equation*} 
Rewriting this relation in two different ways, we have 
\begin{align*}
H_F &= e^{F-\T F} H + (2+V)(1-e^{F-\T F})+(e^{F-\T F}-e^{F-\T^* F})S^*, \\
H_F &= e^{F-\T^* F} H + (2+V)(1-e^{F-\T^* F})+(e^{F-\T^* F}-e^{F-\T F})S.
\end{align*}
Adding these two relations gives
\begin{equation}
\label{johnnygood}
2H_F = 2C_F  H + (2+V)(2-2C_F)+(e^{F-\T F}-e^{F-\T^* F})(S^*-S).
\end{equation}
We further develop the third term on the right side:
\begin{align*}
(e^{F-\T F}-e^{F-\T^* F})(S^*-S) &= (\T\varphi_r - \T^* \varphi_{\ell})(S^*-S) \\
&= (\T \varphi_r - \varphi_r)(S^*-S) + (\varphi_{\ell}- \T^* \varphi_{\ell})(S^*-S) + (\varphi_r-\varphi_{\ell})(S^*-S) \\
&= (\T \varphi_r - \varphi_r)(S^*-S) + (\varphi_{\ell}- \T^* \varphi_{\ell})(S^*-S) \\
&\quad + (g_r-g_{\ell})A' -2^{-1}(g_r-g_{\ell})(S^*+S) + (\varphi_r-\varphi_{\ell})\mathbf{1}_{\{n=0\}}(S^*-S).
\end{align*}
Here, $\mathbf{1}_B$ is the projector onto $B\subset \Z$. Note that $(\varphi_r-\varphi_{\ell})\mathbf{1}_{\{n=0\}} = 0$, and thus \eqref{gogog} is shown. 
\qed
\end{proof}

We are now ready to prove the main result concerning the one-dimensional operator $H$:
\vspace{0.5cm}

\noindent\textit{Proof of Theorem \ref{HunGun1}, the first part}. \quad We first handle the case $E\neq 2$. Suppose that the statement of the theorem is false. Then $\theta_E = \theta_E(\alpha_0) = (E-2)/\cosh(\alpha_0) + 2 \in \Theta(H) \setminus \{+2\}$ for some $\alpha_0 \in [0,\infty)$, and there is an interval 
\begin{equation}
\Sigma_0:= (\theta_E(\alpha_0)-2\delta,\theta_E(\alpha_0)+2\delta)
\end{equation}
such that the Mourre estimate holds there, i.e.\
\begin{equation}
\label{esoi9}
E_{\Sigma_0}(H)[H,A']_{\circ}E_{\Sigma_0}(H) \geqslant \eta E_{\Sigma_0}(H) + K
\end{equation}
for some $\eta >0$ and some compact operator $K$. For the remainder of the proof, $\delta$, $\eta$ and $K$ are fixed. If $\alpha_0 > 0$, choose $\alpha_1>0$ and $\gamma>0$ such that 
\begin{equation}
\alpha_1 < \alpha_0 < \alpha_1 +\gamma.
\end{equation} 
If however $\alpha_0=0$, let $\alpha_1=0$ and $\gamma>0$. By continuity of the map $\theta_E(\alpha) = (E-2)/\cosh(\alpha)+2$, $\theta_E(\alpha_1) \to \theta_E(\alpha_0)$ as $\alpha_1 \to \alpha_0$, so taking $\alpha_1$ close enough to $\alpha_0$ we obtain intervals
\begin{equation*}
\Sigma_1:= (\theta_E(\alpha_1)-\delta,\theta_E(\alpha_1)+\delta) \subset \Sigma_0
\end{equation*}
with the inclusion remaining valid as $\alpha_1 \to \alpha_0$. Multiplying to the right and left of \eqref{esoi9} by $E_{\Sigma_1}(H)$, we obtain
\begin{equation}
\label{esoi10}
E_{\Sigma_1}(H)[H,A']_{\circ}E_{\Sigma_1}(H) \geqslant \eta E_{\Sigma_1}(H) + E_{\Sigma_1}(H)KE_{\Sigma_1}(H).
\end{equation}
Later in the proof $\alpha_1$ will be taken even closer to $\alpha_0$ allowing $\gamma$ to be as small as necessary in order to lead to a contradiction (in this limiting process, $\delta$, $\eta$ and $K$ are fixed). Before delving into the details of the proof, we expose the strategy. For a suitable sequence of functions $\{F_s(x)\}_{s>0}$, let
\begin{equation}
\label{lsl44}
\Psi_s := e^{F_s}\psi/ \| e^{F_s}\psi\|.
\end{equation}
With $F_s$ and $\Psi_s$ instead of $F$ and $\psi_F$ respectively, we apply Proposition \ref{propdodo} to conclude that 
\begin{equation}
\label{okyt}
\limsup \limits_{s \downarrow 0} \ \langle \Psi_s, [H,A']_{\circ} \Psi_s \rangle \leqslant \limsup \limits_{s \downarrow 0} | \langle \Psi_s, 2^{-1}\big(M_{F_s}+G_{F_s}\big) \Psi_s \rangle |.
\end{equation}
Notice how the the negativity of the first two terms on the right side of \eqref{theZerothone} was crucial. We have also written $M_{F_s}$ and $G_{F_s}$ instead of $M$ and $G$ to show the dependence on $F_s$. The first part of the proof consists in showing that 
\begin{equation}
\label{bighome}
\limsup \limits_{s \downarrow 0} \ \langle \Psi_s, [H,A']_{\circ} \Psi_s \rangle \leqslant \limsup \limits_{s \downarrow 0} |  \langle \Psi_s, 2^{-1}\big(M_{F_s}+G_{F_s}\big) \Psi_s \rangle | \leqslant c \epsilon_{\gamma} 
\end{equation}
for some $\epsilon_{\gamma}>0$ satisfying $\epsilon_{\gamma} \to 0$ when $\gamma \to 0$. Here and thereafter, $c >0$ denotes a constant independent of $s$, $\alpha_1$ and $\gamma$. The second part of the proof consists in showing that 
\begin{equation}
\label{keytwo2}
\limsup \limits_{s \downarrow 0} \|(H - \theta_E(\alpha_1)) \Psi_s \| \leqslant  c \epsilon_{\gamma}.
\end{equation}
Roughly speaking \eqref{keytwo2} says that $\Psi_s$ has energy concentrated about $\theta_E(\alpha_1)$ and so localizing \eqref{bighome} about this energy will lead to \begin{equation}
\label{dldl9}
\limsup \limits_{s \downarrow 0} \ \langle \Psi_s, E_{\Sigma_1}(H) [H,A']_{\circ} E_{\Sigma_1}(H) \Psi_s \rangle \leqslant c\epsilon_{\gamma}.
\end{equation}
However, the Mourre estimate \eqref{esoi10} holds on $\Sigma_1$. In the end, the contradiction will come from the fact that the Mourre estimate asserts that the left side of \eqref{dldl9} is not that small.

We now begin in earnest the proof. Notice that $\psi \in \mathcal{D}(\vartheta_{\alpha_1})$ but $\psi \not\in \mathcal{D}(\vartheta_{\alpha_1+\gamma})$. Let $\Upsilon_s$ be the interpolating function defined in \eqref{Upsilon}, and for $s>0$ let
\begin{equation}
\label{hyy9}
F_s(x) := \alpha_1 x + \gamma \Upsilon_s(x).
\end{equation}
As explained in the multi-dimensional case, $F_s$ induces a radial potential as follows : $(F_s u)(n) := F_s (\langle n \rangle)u(n)$, for all $u \in \ell^2(\Z)$. By \eqref{espn1gn}, $e^{F_s}\psi \in \ell^2(\Z)$ for all $s>0$, but $\|e^{F_s} \psi \| \to \infty$ as $s \downarrow 0$. To ease the notation, we will be bounding various quantities by the same constant $c>0$, a constant that is independent of $\alpha_1$, $\gamma$, $s$ and of position $x$ (or $n$).


\noindent \textbf{Part 1.}  We use Proposition \ref{propdodo} with $F_s$ replacing $F$, and so we verify that $F_s$ satisfies the hypotheses of that proposition. Since
\begin{equation*}
F_s'(x) = \alpha_1 + \gamma\Upsilon_s'(x) \quad \text{and} \quad
F_s''(x) = \gamma\Upsilon_s''(x),
\end{equation*}
indeed $|F_s'(x)| \leqslant c$, $|xF_s''(x)| \leqslant c$. Dividing \eqref{theZerothone} by $\|e^{F_s}\psi\|^2$ throughout we obtain \eqref{okyt} as claimed. To prove \eqref{bighome}, we need two ingredients. First, for any bounded set $B \subset \Z$,
\begin{equation}
\label{okkk0}
\lim \limits_{s \downarrow 0} \sum_{n\in B} |\Psi_s(n)|^2=0.
\end{equation}
In particular, $\Psi_s$ converges weakly to zero. What's more, we also have for any $k \in \N$
\begin{equation}
\label{okkk9}
\lim \limits_{s \downarrow 0} \sum_{n\in B} |(S^k\Psi_s)(n)|^2=0, \ \ \ \text{and} \ \ \  
\lim \limits_{s \downarrow 0} \sum_{n\in B} |((S^*)^k\Psi_s)(n)|^2=0.
\end{equation}
Now $M_{F_s}$ and $G_{F_s}$ are a finite sum of terms of the form $P_1(S,S^*)TP_2(S,S^*)$, where $P_1$ and $P_2$ are polynomials and the $T=T(n)$ are sequences. The second item to show is that,
\begin{equation}
\label{m+q}
|T(n)| \leqslant c(\langle n \rangle ^{-1} + \epsilon_{\gamma}).
\end{equation}
In other words we want smallness coming from decay in position $n$ or from $\gamma$. Outside a sufficiently large bounded set, decay in position can be converted into smallness in $\gamma$ by using \eqref{okkk0} while $P_1(S,S^*)$ and $P_2(S,S^*)$ get absorbed in the process thanks to \eqref{okkk9}. Consider first $M=M_{F_s}$ given by \eqref{FunnyGuy}. Applying the Mean Value Theorem (MVT) gives the uniform estimates in $s$
\begin{equation}
\label{hock}
|\T F_s - F_s| \quad \text{and} \quad |\T^* F_s - F_s| \in O(1).
\end{equation} It follows that 
\begin{equation*}
|\varphi_{\ell}| \quad \text{and} \quad |\varphi_r| \in O(1), \quad \text{and} \quad
|g_r-g_{\ell}| \in O(\langle n \rangle^{-1}).
\end{equation*}
To handle the term $(\T^*\varphi_{\ell}-\varphi_{\ell})$, define the function $f(x):=e^{F_s(\langle x-1 \rangle )-F_s(\langle x \rangle)}$. Then $(\T^*\varphi_{\ell}-\varphi_{\ell})(n) = f(n+1) - f(n)$. Applying twice the MVT gives
\begin{equation*}
|(\T^*\varphi_{\ell}-\varphi_{\ell})(n)| \leqslant c(\langle n \rangle ^{-3} + \gamma \langle n \rangle ^{-1}).
\end{equation*}
The same estimate holds for the similar terms like $(\varphi_r-\T\varphi_r)$, $(\T^*\varphi_r-\varphi_r)$ and so forth. We turn our attention to $G=G_{F_s}$ given by \eqref{goHunt}. By \eqref{hock},
$|W_{F_s}| \in O(1)$. To estimate $(W_{F_s}-W_{\T^*F_s})$, let $g(x) := \sqrt{\cosh(F_s(\langle x-1 \rangle)-F_s(\langle x \rangle))-1}$, so that  $(W_{F_s}-W_{\T^*F_s})(n) = g(n)-g(n+1)$. Moreover,
\begin{equation*}
g'(x) = \frac{(F_s'(\langle x-1\rangle)-F_s'(\langle x \rangle))\sinh(F_s(\langle x-1 \rangle)-F_s(\langle x \rangle))}{2\sqrt{\cosh(F_s(\langle x-1\rangle)-F_s(\langle x\rangle))-1}}.
\end{equation*}
If $\alpha_1 > 0$, then $|F_s(\langle x-1\rangle)-F_s(\langle x\rangle)| \geqslant c'\alpha_1$ for some constant $c'>0$ independent of $x$ and $s$, and so $\cosh(F_s(\langle x-1 \rangle)-F_s(\langle x\rangle))-1$ is uniformly bounded from below by a positive number. Applying the MVT to $(F_s'(\langle x-1\rangle)-F_x'(\langle x\rangle))$ yields the estimate 
\begin{equation*}
|(W_{F_s}-W_{\T^*F_s})(n)| \leqslant c(\langle n \rangle ^{-3} + \gamma \langle n \rangle ^{-1}).
\end{equation*}
If however $\alpha_1 = 0$, then 
\begin{equation}
\label{dks}
|(\T F_s - F_s)(n) - (F_s-\T^*F_s)(n)| \leqslant c \gamma \langle n \rangle ^{-1}.
\end{equation} 
By continuity of the function $x \mapsto \sqrt{\cosh(x)-1}$ we have that for any $\epsilon_{\gamma}>0$,
\begin{equation*}
|W_{F_s}-W_{\T^* F_s}| = |\sqrt{\cosh(\T F_s-F_s)-1}-\sqrt{\cosh( F_s-\T^*F_s)-1}| \leqslant \epsilon_{\gamma}
\end{equation*}
whenever \eqref{dks} holds. A similar argument works for $(W_{F_s}-W_{\T F_s})$. Thus \eqref{m+q} is proven, and this shows \eqref{bighome} when combined with the fact that $\Psi_s$ converges weakly to zero.


\noindent \textbf{Part 2.} We now prove \eqref{keytwo2}. Consider Lemma \ref{Goodlemma} with $F_s$ instead of $F$. We claim that 
\begin{equation}
\label{keyone}
\lim \limits_{s \downarrow 0} \Big\|\left( C_{F_s} H + 2-E -2C_{F_s} \right) \Psi_s \Big\| =0.
\end{equation}
By \eqref{gogog} of Lemma \ref{Goodlemma}, this is equivalent to showing that 
\begin{equation*}
\lim \limits_{s \downarrow 0} \big\| R_{F_s} \Psi_s \big\| =0.
\end{equation*}
Dividing each term in \eqref{theZerothone} by $\|e^{F_s}\psi\|^2$, we see that $\|\sqrt{g_r-g_{\ell}}A' \Psi_s\| \leqslant c$. Let $\chi_N$ denote the characteristic function of the set $\{ n \in \Z : (g_r-g_{\ell}) < N ^{-1}\}$. Then 
\begin{equation*}
\limsup \limits_{s \downarrow 0} \|(g_r-g_{\ell})A' \Psi_s\| \leqslant \limsup \limits_{s \downarrow 0} N^{-\frac{1}{2}}\|\chi_N\sqrt{g_r-g_{\ell}}A' \Psi_s\|+\|(1-\chi_N)(g_r-g_{\ell})A' \Psi_s\| \leqslant cN^{-\frac{1}{2}}.
\end{equation*}
Here we used the fact that $1-\chi_N$ has support in a fixed, bounded set as $s\downarrow 0$. Since $N$ is arbitrary, this shows that $\|(g_r-g_{\ell})A' \Psi_s\| \to 0$ as $s \downarrow 0$. The other terms of $R_{F_s}$ are handled similarly. Note that for the term containing $V$ we use the fact it goes to zero at infinity, and from Part 1, $(\T\varphi_r-\varphi_r)$, $(\varphi_{\ell}-\T^*\varphi_{\ell})$ and $(g_r-g_{\ell})$ also go to zero at infinity. Hence \eqref{keyone} is proved. Let $\kappa := \kappa(n) = \sign(n)$. From the expression of $F'_s$, we have the estimates : 
\begin{equation*}
|(F_s - \T F_s)(n) - \kappa(n)\alpha_1 | \leqslant c(\alpha_1 \langle n\rangle^{-1} + \gamma) \quad \text{and} \quad |(F_s - \T^* F_s)(n) - (-\kappa(n)\alpha_1) | \leqslant c(\alpha_1 \langle n\rangle^{-1} + \gamma).
\end{equation*}
Therefore, outside a fixed bounded set we have 
\begin{equation}
\label{odk}
|(F_s - \T F_s) - \kappa \alpha_1 | \leqslant c\gamma  \ \ \ \text{and} \ \ \ |(F_s - \T^* F_s) - (-\kappa \alpha_1) | \leqslant c\gamma.
\end{equation}
By continuity of the exponential function, we have for any $\epsilon_{\gamma} > 0$ that
\begin{equation*}
|e^{F_s - \T F_s} - e^{\kappa \alpha_1}| \leqslant \epsilon_{\gamma} \ \ \ \text{and} \ \ \ |e^{F_s - \T ^*F_s} - e^{-\kappa \alpha_1}| \leqslant \epsilon_{\gamma}
\end{equation*} 
whenever the respective terms of \eqref{odk} hold. It follows from \eqref{keyone} that 
\begin{equation*}
\limsup \limits_{s \downarrow 0} \big \|\big[ 2^{-1}\left(e^{\alpha_1}+e^{-\alpha_1}\right)H + 2-E -\left(e^{\alpha_1}+e^{-\alpha_1}\right) \big] \Psi_s \big\| \leqslant c\epsilon_{\gamma}.
\end{equation*}
Dividing this expression by $\cosh(\alpha_1)$ proves \eqref{keytwo2}.




\noindent \textbf{Part 3.} By functional calculus and \eqref{keytwo2}, we have
\begin{equation}
\label{dod}
\limsup \limits_{s \downarrow 0} \|E_{\R \setminus \Sigma_1}(H) \Psi_s\| \leqslant \limsup \limits_{s \downarrow 0} \delta^{-1} \big\|E_{\R \setminus \Sigma_1}(H) \big(H - \theta_E(\alpha_1)\big)\Psi_s \big\| \leqslant c\epsilon_{\gamma}.
\end{equation}
We have  
\begin{equation}
\label{ls8}
\langle \Psi_s, E_{\Sigma_1}(H)[H,A']_{\circ}E_{\Sigma_1}(H) \Psi_s \rangle = \langle \Psi_s, [H,A']_{\circ} \Psi_s \rangle - f_1(s) - f_2(s), \quad \text{where}
\end{equation}
\begin{equation*}
f_1(s) = \langle \Psi_s, E_{\R \setminus \Sigma_1}(H)[H,A']_{\circ}E_{\Sigma_1}(H) \Psi_s \rangle, \quad \text{and} \quad
f_2(s) = \langle \Psi_s, [H,A']_{\circ}E_{\R \setminus \Sigma_1}(H) \Psi_s \rangle.
\end{equation*}
By \eqref{dod},
\begin{equation*}
\max_{i=1,2} \limsup \limits_{s \downarrow 0} |f_i(s)| \leqslant c \epsilon_{\gamma}.
\end{equation*}
This together with \eqref{bighome} and \eqref{ls8} implies
\begin{equation}
\label{sls9}
\limsup \limits_{s \downarrow 0} \ \langle \Psi_s, E_{\Sigma_1}(H)[H,A']_{\circ}E_{\Sigma_1}(H) \Psi_s \rangle \leqslant c\epsilon_{\gamma}.
\end{equation}
On the other hand, by the Mourre estimate \eqref{esoi10}, we have that
\begin{equation}
\label{esoi11}
\langle \Psi_s, E_{\Sigma_1}(H)[H,A']_{\circ}E_{\Sigma_1}(H) \Psi_s \rangle \geqslant \eta \|E_{\Sigma_1}(H) \Psi_s \|^2 + \langle \Psi_s, E_{\Sigma_1}(H)KE_{\Sigma_1}(H) \Psi_s \rangle.
\end{equation}
Thus, since $\Psi_s$ converges weakly to zero and $E_{\Sigma_1}(H)KE_{\Sigma_1}(H)$ is compact, we have, using \eqref{dod}
\begin{equation}
\label{eq7893}
\liminf \limits_{s \downarrow 0} \ \langle \Psi_s, E_{\Sigma_1}(H) [H,A']_{\circ}E_{\Sigma_1}(H)\Psi_s \rangle \geqslant \eta (1-c\epsilon_{\gamma}^2).
\end{equation}
Recall that $\epsilon_{\gamma} \to 0$ as $\gamma \to 0$. Taking first $\alpha_1$ sufficiently close to $\alpha_0$, we can then take $\gamma$ small enough to see that \eqref{eq7893} contradicts \eqref{sls9}. The proof is complete for the case $E \neq 2$.




\noindent\textbf{Part 4.} Case $E=2$: the proof is almost the same as before but a bit simpler. We briefly go over the proof to point out the small adjustments. Assuming the statement of the theorem to be false, we have that $2 \in \Theta(H)$, and also that $\psi \not\in \mathcal{D}(\vartheta_{\alpha})$ for some $\alpha \in (0,\infty)$. Since $\Theta(H)$ is open, there is an interval 
\begin{equation*}
\Sigma := (2-\delta,2+\delta)
\end{equation*}
such that the Mourre estimate holds there, i.e.
\begin{equation}
E_{\Sigma}(H)[H,A']_{\circ}E_{\Sigma}(H) \geqslant \eta E_{\Sigma}(H) + K
\end{equation}
for some $\eta>0$ and some compact operator $K$. Let $\alpha_0 := \inf \{ \alpha \geqslant 0 : \psi \not\in \mathcal{D}(\vartheta_{\alpha})\}$. As before, let $\alpha_1$ and $\gamma$ be such that $\alpha_1 < \alpha_0 < \alpha_1 + \gamma$ if $\alpha_0 >0$; if $\alpha_0 = 0$, let $\alpha_1 =0$. Let $F_s$ and $\Psi_s$ be defined as before (see \eqref{hyy9} and \eqref{lsl44}), so that $\Psi_s$ has norm one but converges weakly to zero. The calculation of Part 1 shows that 
\begin{equation*}
\limsup \limits_{s \downarrow 0} \ \langle \Psi_s, [H,A']_{\circ} \Psi_s \rangle \leqslant c\epsilon_{\gamma},
\end{equation*}
whereas the calculation of Part 2 shows that 
\begin{equation*}
\lim \limits_{s \downarrow 0} \|(H-2)\Psi_s\| \leqslant c\epsilon_{\gamma}.
\end{equation*}
The functional calculus then gives 
\begin{equation*}
\limsup \limits_{s \downarrow 0} \|E_{\R \setminus \Sigma(H)}\Psi_s\| \leqslant \limsup \limits_{s\downarrow 0} \delta^{-1} \|E_{\R \setminus \Sigma}(H)(H-2)\Psi_s\| \leqslant c \epsilon_{\gamma}.
\end{equation*}
As in Part 3, we get inequalities \eqref{sls9} and \eqref{eq7893} with $\Sigma$ instead of $\Sigma_1$. Taking $\alpha_1$ very close to $\alpha_0$ in order to take $\gamma$ sufficiently small, these two inequalities disagree. The proof is complete.
\qed









It remains to show however that 
\begin{equation}
H\psi = E\psi, \quad \text{and} \quad \psi \in \mathcal{D}(\vartheta_{\alpha}) \quad \text{for all} \quad \alpha \geqslant 0 \quad \text{implies} \quad \psi=0.
\end{equation} 
We slightly modify the notation we have been using so far. Let 
\begin{equation}
F_{\alpha}(n) := \alpha |n| \quad \text{and} \quad \psi_{\alpha}(n) := e^{F_{\alpha}(n)}\psi(n) = e^{\alpha|n|}\psi(n), \quad \text{for all} \ n \in \Z.
\end{equation}

\noindent \textit{Proof of Theorem \ref{HunGun1}, the second part}. \quad The proof is by contradiction, and the strategy is as follows: we assume that $\psi \neq 0$ and define $\Psi_{\alpha} := \psi_{\alpha}/\|\psi_{\alpha}\|$. It is not hard to see that $\Psi_{\alpha}$ converges weakly to zero as $\alpha \to + \infty$ (use the fact that the difference equation $H\psi = E\psi$ implies $\psi(n) \neq 0$ infinitely often). In the first part we apply Proposition \ref{propdodo} with $F_{\alpha}$ replacing $F$. In this case we can exactly compute terms to show that 
\begin{multline}
\label{sleyy}
0 = \cosh(\alpha)^{-1}\langle \Psi_{\alpha}, [V,A']_{\circ} \Psi_{\alpha} \rangle  + 2\tanh(\alpha)\|\sqrt{|N|}(S^*-S)\Psi_{\alpha}\|^2 \\
+ \|\sqrt{\Delta(4-\Delta)} \Psi_{\alpha}\|^2 - \tanh(\alpha) \left( 2\Psi_{\alpha}^2(0) + \left(\Psi_{\alpha}(-1)-\Psi_{\alpha}(1)\right)^2 \right).
\end{multline}
In the second part, we apply Lemma \ref{Goodlemma} again with $F_{\alpha}$ replacing $F$. We show that 
\begin{equation}
\label{sisz}
\lim \limits_{\alpha \to +\infty} \|\sqrt{\Delta(4-\Delta)}\Psi_{\alpha}\|^2 = \lim \limits_{\alpha \to +\infty} \Re \ \langle \Psi_{\alpha}, \Delta(4-\Delta) \Psi_{\alpha} \rangle = 2.
\end{equation}
The conclusion is then imminent: taking the limit $\alpha \to + \infty$ in \eqref{sleyy}, and recalling that $[V,A']_{\circ}$ exists as a bounded operator and $\Psi_{\alpha}$ converges weakly to zero leads to a contradiction.

\noindent \textbf{Part 1.} It follows from \eqref{zinyDDD} and the limiting argument of Proposition \ref{propdodoD} that 
\begin{equation*}
\langle \psi_{\alpha}, [H, A']_{\circ} \psi_{\alpha} \rangle = \langle \psi_{\alpha}, A' [e^F,\Delta] e^{-F} \psi_{\alpha} \rangle + \langle \psi_{\alpha}, e^{-F} [e^F,\Delta] A' \psi_{\alpha} \rangle. 
\end{equation*}
All terms are computed exactly:
\begin{equation}
e^{(\T F_{\alpha}-F_{\alpha})(n)} = \begin{cases}
e^{-\alpha} &\text{if $n \geqslant 1$}\\
e^{\alpha} &\text{if $n \leqslant 0$}
\end{cases} \hspace{1cm} \text{and} \hspace{1cm} 
e^{(\T^* F_{\alpha}-F_{\alpha})(n)}= \begin{cases}
e^{\alpha} &\text{if $n \geqslant 0$}\\
e^{-\alpha} &\text{if $n \leqslant -1$},
\end{cases}
\end{equation}
\begin{equation}
\label{ty8777}
e^{(F_{\alpha}-\T F_{\alpha})(n)} = \begin{cases}
e^{\alpha} &\text{if $n \geqslant 1$}\\
e^{-\alpha} &\text{if $n \leqslant 0$}
\end{cases} \hspace{1cm} \text{and} \hspace{1cm} 
e^{(F_{\alpha}-\T^* F_{\alpha})(n)}= \begin{cases}
e^{-\alpha} &\text{if $n \geqslant 0$}\\
e^{\alpha} &\text{if $n \leqslant -1$}.
\end{cases}
\end{equation}
Let $\mathbf{1}_B$ be the projector onto $B \subset \Z$. Therefore
\begin{align*}
\vr-\vl &= 2\sinh(\alpha) \text{sign}(N) \mathbf{1}_{\{n \neq 0\}}, \ &
\vr+\vl &= 2\left(\cosh(\alpha) - 1 + \sinh(\alpha)\mathbf{1}_{\{n=0\}}\right), \\
\T^*\varphi_{\ell}-\varphi_{\ell} &= -2\sinh(\alpha)\mathbf{1}_{\{n=0\}}, \ &
\varphi_{\ell} - \T\varphi_{\ell} &= -2\sinh(\alpha)\mathbf{1}_{\{n=+1\}}, \\ 
\T \varphi_r - \varphi_r &= -2\sinh(\alpha)\mathbf{1}_{\{n=0\}}, \ & \varphi_r - \T^* \varphi_r &= -2\sinh(\alpha)\mathbf{1}_{\{n=-1\}}, \\ 
\T^*\varphi_{\ell}-{\T^*}^2\varphi_{\ell} &= 2\sinh(\alpha)\mathbf{1}_{\{n=-1\}}, \ &
\T \varphi_r - \T^2 \varphi_r &= 2\sinh(\alpha)\mathbf{1}_{\{n=+1\}}.
\end{align*}
Let $\mathcal{T} := A' [e^F,\Delta] e^{-F} + e^{-F} [e^F,\Delta] A'$. By \eqref{commutatorSD} and \eqref{commutatorS*D}, we have:
\begin{align*} 
\mathcal{T} &= A' (-S e^{F} \vr -S^* e^F \vl ) e^{-F}  + e^{-F} (\vr e^{F} S^* + \vl e^F S ) A' \\
&= - A' (S \vr + S^* \vl) + (\vr S^* + \vl S)A'.
\end{align*}
Plug in $A' = 2^{-1}(S^*+S) +N(S^*-S)$ and simplify to get $\mathcal{T} = \mathcal{T}_1 + \mathcal{T}_2$, where
\begin{equation*}
\mathcal{T}_1 := 2^{-1} \left( -S^2 \vr + 3 \vr (S^*)^2 - (S^*)^2 \vl +3  \vl S^2 \right) - (\vr + \vl), \quad \text{and}
\end{equation*}
\begin{equation*}
\mathcal{T}_2 := N \left( S^2 \vr + \vr (S^*)^2 - (S^*)^2 \vl - \vl S^2 \right) - 2N(\vr - \vl).
\end{equation*}
We calculate $\mathcal{T}_1$:
\begin{align*}
\mathcal{T}_1 &= -2^{-1}\left( \vr + \vl \right) \left(2 - S^2 - (S^*)^2 \right) -2^{-1} \left( S^2 \vr + \vr S^2 + (S^*)^2\vl + \vl (S^*)^2 \right) + \vr (S^*)^2 + \vl S^2 \\
&= -(\cosh(\alpha)-1+\sinh(\alpha)\mathbf{1}_{\{n=0\}})\Delta(4-\Delta) + (\vr-\vl)\left((S^*)^2 - S^2 \right) \\
&\quad  +2^{-1}\left( (\vl - \T^* \vl)(S^*)^2 + (\T^* \vl - \T^{*2} \vl)(S^*)^2 +(\vr - \T \vr)S^2 + (\T \vr -\T^2 \vr)S^2 \right) \\
&= \mathcal{T}_{1;1} + \mathcal{T}_{1;2}
 \end{align*}
 where
\begin{equation}
\label{T11}
\mathcal{T}_{1;1} :=  -(\cosh(\alpha)-1)\Delta(4-\Delta),\quad \text{and}
\end{equation}
\begin{align*}
\mathcal{T}_{1;2} &:=  -\sinh(\alpha)\mathbf{1}_{\{n=0\}}\Delta(4-\Delta) +2\sinh(\alpha)\text{sign}(N)\mathbf{1}_{\{n\neq 0\}}((S^*)^2-S^2) \\
 &\quad +\sinh(\alpha)\left(\mathbf{1}_{\{n=0\}}(S^*)^2+\mathbf{1}_{\{n=-1\}}(S^*)^2+\mathbf{1}_{\{n=0\}}S^2+\mathbf{1}_{\{n=1\}}S^2 \right).
\end{align*}
We calculate $\mathcal{T}_2$: 
\begin{align*}
\mathcal{T}_2 &= -N(\vr -\vl) \left(2 - S^2 - (S^*)^2\right) + N(S^2\vr - \vr S^2) +N(\vl (S^*)^2 - (S^*)^2 \vl) \\
&= -N(\vr -\vl) \Delta (4- \Delta) + N \left(\T^2 \vr - \T \vr + \T \vr - \vr \right) S^2 + N \left(\vl - \T^* \vl + \T^* \vl - \T^{*2} \vl \right) (S^*)^2 \\
&= -2\sinh(\alpha) |N| \Delta(4-\Delta) + 2\sinh(\alpha) N \left( -(\mathbf{1}_{\{n=1\}} +\mathbf{1}_{\{n=0\}})S^2 + (\mathbf{1}_{\{n=0\}}+\mathbf{1}_{\{n=-1\}})(S^*)^2\right) \\
&= -2\sinh(\alpha) |N| \Delta(4-\Delta) - 2\sinh(\alpha) \left( \mathbf{1}_{\{n=1\}} S^2 + \mathbf{1}_{\{n=-1\}}(S^*)^2\right).
\end{align*}
The following commutation formulae hold
\begin{equation}
\label{ei87}
S^*(\mathbf{1}_{\{n\neq 0\}}\sign(N)) = [\mathbf{1}_{\{n\neq 0\}}\sign(N) + \mathbf{1}_{\{n=0\}} + \mathbf{1}_{\{n=-1\}}]S^*,
\end{equation}
\begin{equation}
\label{ei88}
S(\mathbf{1}_{\{n\neq 0\}}\sign(N)) = [\mathbf{1}_{\{n\neq 0\}}\sign(N) - \mathbf{1}_{\{n=0\}} - \mathbf{1}_{\{n=+1\}}]S.
\end{equation}
Using
\begin{equation*}
S |N| = |N| S + \left( \mathbf{1}_{\{n=0\}} - \mathbf{1}_{\{n \neq 0\}}\text{sign}(N) \right) S,
\end{equation*}
\begin{equation*}
S^* |N| = |N| S^* + \left( \mathbf{1}_{\{n=0\}} +\mathbf{1}_{\{n \neq 0\}}\text{sign}(N) \right) S^*,
\end{equation*}
one checks that
\begin{equation}
\label{RelateTOme}
|N| \Delta(4 -\Delta) = (S-S^*) |N|(S^*-S) - \mathbf{1}_{\{n=0\}} \Delta (4-\Delta) - \mathbf{1}_{\{n \neq 0\}}\text{sign}(N) (S^2-(S^*)^2). 
\end{equation}
Therefore $\mathcal{T}_2 = \mathcal{T}_{2;1} + \mathcal{T}_{2;2}$, where 
\begin{equation}
\label{T21}
\mathcal{T}_{2;1} =  -2\sinh(\alpha)(S-S^*) |N|(S^*-S), \quad \text{and}
\end{equation} 
\begin{align*}
\mathcal{T}_{2;2} &= - 2\sinh(\alpha) \left( \mathbf{1}_{\{n=1\}} S^2 + \mathbf{1}_{\{n=-1\}}(S^*)^2\right) \\
& \quad + 2\sinh(\alpha) \left( \mathbf{1}_{\{n=0\}} \Delta (4-\Delta) + \mathbf{1}_{\{n \neq 0\}}\text{sign}(N) (S^2-(S^*)^2)  \right).
\end{align*}
Finally, a calculation shows that
\begin{equation}
\label{Tsum}
\mathcal{T}_{1;2} + \mathcal{T}_{2;2} =  \sinh(\alpha)\left( 2\mathbf{1}_{\{n=0\}} - \mathbf{1}_{\{n=-1\}}(S^*)^2 -\mathbf{1}_{\{n=1\}}S^2 \right).
\end{equation}
Note that 
\begin{equation*}
\langle \psi_{\alpha}, [H, A']_{\circ} \psi_{\alpha} \rangle = \langle \psi_{\alpha}, \mathcal{T} \psi_{\alpha} \rangle = \langle \psi_{\alpha}, \left(\mathcal{T}_{1;1} + \mathcal{T}_{2;1} +  \mathcal{T}_{1;2} + \mathcal{T}_{2;2} \right) \psi_{\alpha} \rangle.
\end{equation*}
Plugging in for $\mathcal{T}_{1;1}$, $\mathcal{T}_{2;1}$ and $\mathcal{T}_{1;2} + \mathcal{T}_{2;2}$ given by \eqref{T11}, \eqref{T21} and \eqref{Tsum} yields 
\begin{align*}
\langle \psi_{\alpha}, [H,A']_{\circ} \psi_{\alpha} \rangle &= - 2\sinh(\alpha)\|\sqrt{|N|}(S^*-S)\psi_{\alpha}\|^2 - (\cosh(\alpha)-1) \langle \psi_{\alpha}, \Delta(4-\Delta) \psi_{\alpha} \rangle \\
&\quad +\sinh(\alpha) \left( 2\psi_{\alpha}^2(0) + \left(\psi_{\alpha}(-1)-\psi_{\alpha}(1)\right)^2 \right).
\end{align*}
Cancelling $\langle \psi_{\alpha}, [\Delta,A'] \psi_{\alpha} \rangle = \langle \psi_{\alpha}, \Delta(4-\Delta) \psi_{\alpha} \rangle$ on both sides and dividing throughout by $\cosh(\alpha)\|\psi_{\alpha}\|^2$ yields \eqref{sleyy} as required.

\noindent \textbf{Part 2.} 
From \eqref{ty8777},
\begin{equation*}
2^{-1}(e^{F_{\alpha}-\T F_{\alpha}} + e^{F_{\alpha}-\T^* F_{\alpha}})  = \begin{cases}
\cosh(\alpha) &\text{if $|n| \geqslant 1$}\\
e^{-\alpha} &\text{if $n = 0$},
\end{cases}
\end{equation*}
\begin{equation*}
2^{-1}(e^{F_{\alpha}-\T F_{\alpha}} - e^{F_{\alpha}-\T^* F_{\alpha}})  = \sinh(\alpha)\mathbf{1}_{\{n \neq 0\}}\sign(N).
\end{equation*}
We apply \eqref{johnnygood} of Lemma \ref{Goodlemma}:
\begin{align*}
H_{F_{\alpha}} &= \cosh(\alpha) \Delta + \mathbf{1}_{\{n=0\}}(e^{-\alpha}-\cosh(\alpha))\Delta + V + 2(1-\cosh(\alpha)) \\ 
&\quad + 2\mathbf{1}_{\{n=0\}}(\cosh(\alpha)-e^{-\alpha}) + \sinh(\alpha)\mathbf{1}_{\{n \neq 0\}}\sign(N)(S^*-S).
\end{align*}
The goal is to square $H_{F_{\alpha}}$. Divide throughout by $\cosh(\alpha)$ and let $c_{\alpha} := (e^{-\alpha}\cosh(\alpha)^{-1}-1)$:
\begin{align}
\label{eis990}
\cosh(\alpha)^{-1}H_{F_{\alpha}} &= \Delta + c_{\alpha} \mathbf{1}_{\{n=0\}}\Delta + \cosh(\alpha)^{-1} V + 2(\cosh(\alpha)^{-1}-1) - 2c_{\alpha}\mathbf{1}_{\{n=0\}} \\ 
 & \quad  + \tanh(\alpha)\mathbf{1}_{\{n \neq 0\}}\sign(N)(S^*-S). \label{sis72}
\end{align}
Note that $\sup_{\alpha \geqslant 0} |c_{\alpha}| \leqslant 2$. Since $(S^*-S)$ is antisymmetric, by \eqref{ei87} and \eqref{ei88}, we see that $\mathbf{1}_{\{n \neq 0\}}\sign(N)(S^*-S)$ is antisymmetric up to a couple of rank one projectors. The same goes for 
$\Delta\mathbf{1}_{\{n \neq 0\}}\sign(N)(S^*-S)$ and $\mathbf{1}_{\{n \neq 0\}}\sign(N)(S^*-S)\Delta$.
Therefore
\begin{equation*}
\lim \limits_{\alpha \to +\infty} \Re \  \langle \Psi_{\alpha}, [\mathbf{1}_{\{n \neq 0\}}\sign(N)(S^*-S)] \Psi_{\alpha} \rangle = 0,
\end{equation*}
\begin{equation*}
\lim \limits_{\alpha \to +\infty} \Re \  \langle \Psi_{\alpha}, \Delta[\mathbf{1}_{\{n \neq 0\}}\sign(N)(S^*-S)] \Psi_{\alpha} \rangle = 0,
\end{equation*}
\begin{equation*}
\lim \limits_{\alpha \to +\infty} \Re \  \langle \Psi_{\alpha}, [\mathbf{1}_{\{n \neq 0\}}\sign(N)(S^*-S)]\Delta \Psi_{\alpha} \rangle = 0.
\end{equation*}
We compute $\big[\tanh(\alpha)\mathbf{1}_{\{n \neq 0\}}\sign(N)(S^*-S)\big]^2$ using \eqref{ei87} and \eqref{ei88}:
\begin{equation*}
\eqref{sis72}^2 = \tanh^2(\alpha)\Big[ \mathbf{1}_{\{n \neq 0\}} (S^2+(S^*)^2-2) + \mathbf{1}_{\{n=-1\}}(1-(S^*)^2) + \mathbf{1}_{\{n=+1\}}({S}^2-1) \big].
\end{equation*}
Thus squaring $\cosh(\alpha)^{-1}H_{F_{\alpha}}$ given by \eqref{eis990}-\eqref{sis72} and recalling that $\Delta(4-\Delta) = 2-S^2-(S^*)^2$ we get
\begin{equation*}
\cosh(\alpha)^{-2} H_{F_{\alpha}}^2 = \Delta(\Delta-4) + 4 -\tanh^2(\alpha)\Delta(4-\Delta) + P_{\alpha},
\end{equation*} 
where $P_{\alpha}$ is a bounded operator satisfying
\begin{equation*}
\lim \limits_{\alpha \to \infty} \Re \ \langle \Psi_{\alpha}, P_{\alpha} \Psi_{\alpha} \rangle = 0.
\end{equation*}
Rearranging and recalling that $H_{F_{\alpha}}\Psi_{\alpha} = E\Psi_{\alpha}$ yields \eqref{sisz} as required.
\qed

\section{Appendix : Technical calculations}
\label{appendix}

The Appendix is devoted to proving the key relations \eqref{theFirstoneD} and \eqref{theFirstoned1} that appear in Propositions \ref{propdodoD} and \ref{propdodo} respectively. Recall that for $B \subset \Z^d$, $\mathbf{1}_B$ denotes the projector onto $B$. We start with the proof of the multi-dimensional formula
\begin{equation}
\label{theFirstoneDDD}
\begin{aligned}
\big \langle \phi, [e^F A' e^F,\Delta] \phi \big \rangle &= \big \langle \phi_F, [A', \Delta] \phi_F \big \rangle - 2\big\|\sqrt{g}A' \phi_F \big\|^2 \\
& \quad - \sum_{i=1}^d \big\|\sqrt{\Delta_i (4-\Delta_i)} W_i \phi_F \big\|^2 + 2^{-1} \big \langle \phi_F, (\mathcal{L} + \mathcal{M} + \mathcal{G}) \phi_F \big \rangle,
\end{aligned}
\end{equation}
where $\phi \in \ell_0(\Z^d)$ and $\phi_F := e^F\phi$. To jump to the proof of the $1d$ relation, go to \eqref{theFirstone}. \\
\begin{proof}
It is understood that the operators are calculated and the commutators developed against $\phi \in \ell_0(\Z^d)$, so we omit the $\phi$ for ease of notation. Usual commutation relations give
\begin{equation}
\label{zinyDDD}
[e^F A' e^F, \Delta] = e^F[A', \Delta]e^F + e^FA'[e^F,\Delta] + [e^F,\Delta]A' e^F.
\end{equation}
We now concentrate on the second and third terms on the right side of the latter relation. The goal is to pop out $e^F A' g A' e^F$ and control the remainder. As pointed out in \cite{FH} and \cite{CFKS}, this is the key quantity to single out. The following commutators will be used repeatedly:
\begin{equation}
\label{commutatorSD}
[e^F,S_i] = -(\T_i e^F - e^F)S_i = S_i(\T_i^* e^F-e^F) = -e^F\varphi_{\ell_i}S_i = S_i\varphi_{r_i} e^F,
\end{equation}
\begin{equation}
\label{commutatorS*D}
[e^F,S^*_i] = -(\T_i^* e^F - e^F)S_i^* = S_i^*(\T_i e^F - e^F) = -e^F\varphi_{r_i} S_i^* = S_i^* \varphi_{\ell_i}e^F.
\end{equation}
\textbf{Part 1 : Creating $e^F A' g A' e^F$ in a first way.} We have
\begin{align*}
[e^F, \Delta_i] &= \varphi_{r_i} e^F S_i^* + \varphi_{\ell_i} e^F S_i \\
&= g_{r_i}  N_i e^F S_i^* + \varphi_{r_i} \mathbf{1}_{\{n_i=0\}} e^F S_i^* + \varphi_{\ell_i} e^F S_i  \\
&= g_{r_i} N_i e^F (S_i^*-S_i) + \varphi_{r_i} \mathbf{1}_{\{n_i=0\}} e^F (S_i^*-S_i) + \big( \varphi_{r_i} + \varphi_{\ell_i} \big) e^F S_i \\
&= g_{r_i} N_i (S_i^*-S_i)e^F + g_{r_i} N_i [e^F,(S_i^*-S_i)] + \varphi_{r_i} \mathbf{1}_{\{n_i=0\}} e^F (S_i^*-S_i) + \big( \varphi_{r_i} + \varphi_{\ell_i} \big) e^F S_i \\
&= g N_i (S_i^*-S_i)e^F + (g_{r_i}-g) N_i (S_i^*-S_i)e^F + \varphi_{r_i} \mathbf{1}_{\{n_i=0\}} (S_i^*-S_i) e^F \\
&\quad + \varphi_{r_i} [e^F,(S_i^*-S_i)] + \big( \varphi_{r_i} + \varphi_{\ell_i} \big) e^F S_i \\
&= g A_i' e^F - 2^{-1}g (S_i^*+S_i) e^F + (g_{r_i}-g) N_i (S_i^*-S_i)e^F + \varphi_{r_i} \mathbf{1}_{\{n_i=0\}} (S_i^*-S_i) e^F\\
& \quad + \varphi_{r_i} [e^F,(S_i^*-S_i)] + \big( \varphi_{r_i} + \varphi_{\ell_i} \big) e^F S_i.
\end{align*}
\begin{align*}
[e^F, \Delta_i] &= - S_ie^F\varphi_{r_i} - S_i^*e^F \varphi_{\ell_i}  \\
&= -S_i e^F N_i g_{r_i} -S_i e^F \varphi_{r_i} \mathbf{1}_{\{n_i=0\}} - S_i^*e^F\varphi_{\ell_i}  \\
&= (S_i^*-S_i)e^F N_i g_{r_i} + (S_i^*-S_i) e^F \varphi_{r_i} \mathbf{1}_{\{n_i=0\}} - S_i^*e^F\big( \varphi_{r_i} + \varphi_{\ell_i} \big) \\
&= e^F(S_i^*-S_i) N_i g_{r_i} +[(S_i^*-S_i),e^F]N_i g_{r_i} + (S_i^*-S_i) e^F \varphi_{r_i} \mathbf{1}_{\{n_i=0\}} - S_i^*e^F\big( \varphi_{r_i} + \varphi_{\ell_i} \big) \\
&= e^F(S_i^*-S_i) N_i g + e^F(S_i^*-S_i) N_i (g_{r_i}-g) + e^F(S_i^*-S_i)\varphi_{r_i} \mathbf{1}_{\{n_i=0\}} \\
&\quad - [e^F,(S_i^*-S_i)]\varphi_{r_i} - S_i^*e^F\big( \varphi_{r_i} + \varphi_{\ell_i} \big)\\
&= e^F A_i' g + 2^{-1}e^F(S^*_i+S_i)g + e^F(S_i^*-S_i) N_i (g_{r_i}-g) + e^F(S_i^*-S_i)\varphi_{r_i} \mathbf{1}_{\{n_i=0\}} \\
&\quad - [e^F,(S_i^*-S_i)]\varphi_{r_i} - S_i^*e^F\big( \varphi_{r_i} + \varphi_{\ell_i} \big). 
\end{align*}
Therefore we have obtained
\begin{equation}
\label{FirstDD}
e^FA'[e^F,\Delta] + [e^F,\Delta]A' e^F =  2e^F A' g A' e^F + e^F (L_r + M_r + G_r + H_r)e^F, \quad \text{where}
\end{equation}
\begin{equation*}
L_r :=  \sum_{i,j} A_i'(g_{r_j}-g) N_j (S_j^*-S_j) +  (S_i^*-S_i) N_i (g_{r_i}-g) A_j',
\end{equation*}
\begin{equation*}
M_r := 2^{-1} \sum_{i,j} - A_i' g (S_j^*+S_j) +  (S_i^*+S_i) g A_j', 
\end{equation*}
\begin{align*}
G_r &:= \sum_{i,j} A_i'\left(\varphi_{r_j} [e^F,(S_j^*-S_j)]e^{-F} + \big( \varphi_{r_j} + \varphi_{\ell_j} \big) e^F S_j e^{-F}\right) \\
&\quad - \sum_{i,j} \left(e^{-F} [e^F,(S_i^*-S_i)]\varphi_{r_i} + e^{-F}S_i^*e^F\big( \varphi_{r_i} + \varphi_{\ell_i} \big) \right)A_j', \quad \text{and}
\end{align*}
\begin{equation*}
H_r := \sum_{i,j} A_i'\varphi_{r_j} \mathbf{1}_{\{n_j=0\}} (S_j^*-S_j) + (S_i^*-S_i) \varphi_{r_i} \mathbf{1}_{\{n_i=0\}} A_j.    
\end{equation*}
We split $M_r$ as follows:
$M_r = M_{r;1} + M_{r;2}$, where
\begin{equation*}
M_{r;1} := 2^{-1}\sum_{i \neq j} -A_i' g(S_j^*+S_j) +  (S_i^*+S_i) g A_j',   
\end{equation*}
\begin{align*}
M_{r;2} &:= 2^{-1}\sum_i - A_i' g(S_i^*+S_i)  +  (S_i^*+S_i) g A_i' = M_{r;2;1} + M_{r;2;2}, \quad \text{with}
\end{align*}
\begin{equation*}
M_{r;2;1} :=  2^{-1}\sum_i -A_i' g_{r_i} (S_i^*+S_i) + (S_i^*+S_i) g_{r_i} A_i',
\end{equation*}
\begin{equation*}
M_{r;2;2} :=  2^{-1}\sum_i -A_i' (g-g_{r_i}) (S_i^*+S_i)  + (S_i^*+S_i) (g-g_{r_i}) A_i'.  
\end{equation*}
We calculate $M_{r;1}$ by expanding $A_i'$ and $A_j'$:
\begin{align*}
M_{r;1}&:= 2^{-1}\sum_{i \neq j} -N_i(S_i^*-S_i) g(S_j^*+S_j)  +  (S_i^*+S_i) g N_j(S_j^*-S_j) \\
&= 2^{-1}\sum_{i \neq j} -N_i\big[ (\T_i^*g)S_i^*-(\T_ig)S_i\big](S_j^*+S_j) + \big[(\T_i^*(gN_j))S_i^*+(\T_i(gN_j))S_i\big](S^*_j-S_j) \\
&= \frac{1}{2}\sum_{i \neq j} N_i\big[\T_ig-\T_jg\big]S_iS_j + N_i\big[ \T_j^*g-\T_i^*g\big]S_i^*S_j^* + \big[N_i(\T_jg-\T_i^*g)+N_j(\T_jg-\T_i^*g)\big]S_i^*S_j. ^{(\ddagger_2)}
\end{align*}
Again expanding $A_i'$:
\begin{align*}
M_{r;2;1} &= 2^{-1}\sum_i (S_i^*+S_i)g_{r_i}(S_i^*+S_i) - (S_i^*-S_i)\varphi_{r_i}\mathbf{1}_{\{n_i \neq 0\}}(S_i^*+S_i) \\
&\quad +2^{-1}\sum_i (S_i^*+S_i)\varphi_{r_i}\mathbf{1}_{\{n_i \neq 0\}}(S_i^*-S_i) \\
&= \sum_i 2^{-1}(S_i^*+S_i)g_{r_i}(S_i^*+S_i) 
+ S_i\varphi_{r_i} S_i^* - S_i^*\varphi_{r_i} S_i + \left(S_i^*\varphi_{r_i} \mathbf{1}_{\{n_i = 0\}} S_i - S_i\varphi_{r_i} \mathbf{1}_{\{n_i = 0\}} S_i^*\right) \\
&= M_{r;2;1;1}+ M_{r;2;1;2}, \quad \text{where}
\end{align*}
\begin{equation*}
M_{r;2;1;1}:= \sum_i 2^{-1}(S_i^*+S_i)g_{r_i}(S_i^*+S_i) + (\T_i\varphi_{r_i} - \T_i^*\varphi_{r_i}),^{(\ddagger_1,\ddagger_3)}
\end{equation*}
\begin{equation*}
M_{r;2;1;2} := \sum_i (\T_i^*\varphi_{r_i})\mathbf{1}_{\{n_i=-1\}}-(\T_i\varphi_{r_i})\mathbf{1}_{\{n_i=+1\}}.  
\end{equation*}
We calulate $G_r$. We note that 
\begin{equation}
\label{glgr1}
(\T_i \varphi_{r_i}) \varphi_{\ell_i} = \varphi_{\ell_i} (\T_i \varphi_{r_i}) = -(\T_i \varphi_{r_i} + \varphi_{\ell_i}), \ \ \ \text{and} \ \ \ (\T_i^* \varphi_{\ell_i}) \varphi_{r_i} = \varphi_{r_i} (\T_i^* \varphi_{\ell_i}) = -(\T_i^* \varphi_{\ell_i} + \varphi_{r_i}).
\end{equation}
\begin{align*}
G_r &:=  \sum_{i,j} A_i' \left(\varphi_{r_j} \big( S_j^* \varphi_{\ell_j} - S_j \varphi_{r_j} \big) + \big( \varphi_{r_j} + \varphi_{\ell_j} \big) S_j \big(\varphi_{r_j} + 1 \big) \right) \\
&\quad - \sum_{i,j} \left(\big( -\varphi_{r_i}S_i^* + \varphi_{\ell_i} S_i \big)\varphi_{r_i} + \big(\varphi_{r_i}+1\big) S_i^* \big( \varphi_{r_i} + \varphi_{\ell_i} \big) \right) A_j' \\
&= \sum_{i,j} A_i' \left(S_j^* (\T_j \varphi_{r_j})\varphi_{\ell_j} - S_j (\T_j^*\varphi_{r_j})\varphi_{r_j} + S_j (\T_j^*\varphi_{r_j}) (\varphi_{r_j}+1) + S_j (\T_j^*\varphi_{\ell_j})(\varphi_{r_j}+1) \right) \\
&\quad + \sum_{i,j} \left(\varphi_{r_i}(\T_i^*\varphi_{r_i})S_i^* - \varphi_{\ell_i}(\T_i \varphi_{r_i})S_i   - (\varphi_{r_i}+1)(\T_i^* \varphi_{r_i})S_i^* - (\varphi_{r_i}+1)(\T_i^*\varphi_{\ell_i})S_i^* \right) A_j' \\
&= \sum_{i,j} A_i'\left(- S_j^* (\varphi_{\ell_j}+\T_j \varphi_{r_j}) + S_j(\T_j^* \varphi_{r_j}-\varphi_{r_j}) \right) + \sum_{i,j} \left( (\varphi_{r_i}-\T_i^*\varphi_{r_i})S_i^* + (\varphi_{\ell_i}+\T_i \varphi_{r_i})S_i \right) A_j' \\
&= G_{r;1} + G_{r;2}, \quad \text{where}
\end{align*}
\begin{equation*}
G_{r;1} := \sum_{i,j} A_i' S_j(\T_j^* \varphi_{r_j}-\varphi_{r_j}) +  (\varphi_{r_i}-\T_i^*\varphi_{r_i})S_i^* A_j',^{(\ddagger_3)}  
\end{equation*}
\begin{equation*}
\label{GR2}
G_{r;2} := \sum_{i,j} -A_i' S_j^* (\varphi_{\ell_j}+\T_j \varphi_{r_j}) + (\varphi_{\ell_i}+\T_i \varphi_{r_i})S_i A_j'.
\end{equation*}
To end this section we note that we are left to deal with $L_r + M_{r;2;1;2} + M_{r;2;2} + G_{r;2} + H_r$.\\
\textbf{Part 2 : Creating $e^F A' g A' e^F$ a second way.} We repeat the calculation with a variation.
\begin{align*}
[e^F, \Delta_i] &= \varphi_{\ell_i} e^F S_i + \varphi_{r_i} e^F S_i^* \\
&= g_{\ell_i}  N_i e^F S_i + \varphi_{\ell_i} \mathbf{1}_{\{n_i=0\}} e^F S_i + \varphi_{r_i} e^F S_i^*  \\
&= -g_{\ell_i} N_i e^F (S_i^*-S_i) + \varphi_{\ell_i} \mathbf{1}_{\{n_i=0\}} e^F (S_i-S_i^*) + \big( \varphi_{r_i} + \varphi_{\ell_i} \big) e^F S_i^* \\
&= -g_{\ell_i} N_i (S_i^*-S_i)e^F - g_{\ell_i} N_i [e^F,(S_i^*-S_i)] + \varphi_{\ell_i} \mathbf{1}_{\{n_i=0\}} e^F (S_i-S_i^*) + \big( \varphi_{r_i} + \varphi_{\ell_i} \big) e^F S_i^* \\
&= g N_i (S_i^*-S_i)e^F - (g_{\ell_i}+g) N_i (S_i^*-S_i)e^F + \varphi_{\ell_i} \mathbf{1}_{\{n_i=0\}} (S_i-S_i^*) e^F \\
&\quad - \varphi_{\ell_i} [e^F,(S_i^*-S_i)] + \big( \varphi_{r_i} + \varphi_{\ell_i} \big) e^F S_i^* \\
&= g A_i' e^F - 2^{-1}g (S_i^*+S_i) e^F - (g_{\ell_i}+g) N_i (S_i^*-S_i)e^F + \varphi_{\ell_i} \mathbf{1}_{\{n_i=0\}} (S_i-S_i^*) e^F \\
&\quad - \varphi_{\ell_i} [e^F,(S_i^*-S_i)] + \big( \varphi_{r_i} + \varphi_{\ell_i} \big) e^F S_i^*. 
\end{align*}
\begin{align*}
[e^F, \Delta_i] &= - S_i^*e^F\varphi_{\ell_i} - S_ie^F \varphi_{r_i}  \\
&= -S_i^* e^F N_i g_{\ell_i} -S_i^* e^F \varphi_{\ell_i} \mathbf{1}_{\{n_i=0\}} - S_ie^F\varphi_{r_i}  \\
&= -(S_i^*-S_i)e^F N_i g_{\ell_i} + (S_i-S_i^*) e^F \varphi_{\ell_i} \mathbf{1}_{\{n_i=0\}} - S_ie^F\big( \varphi_{r_i} + \varphi_{\ell_i} \big) \\
&= -e^F(S_i^*-S_i) N_i g_{\ell_i} -[(S_i^*-S_i),e^F]N_i g_{\ell_i} + (S_i-S_i^*) e^F \varphi_{\ell_i} \mathbf{1}_{\{n_i=0\}} - S_ie^F\big( \varphi_{r_i} + \varphi_{\ell_i} \big) \\
&= e^F(S_i^*-S_i) N_i g - e^F(S_i^*-S_i) N_i (g_{\ell_i}+g) + e^F(S_i-S_i^*)\varphi_{\ell_i} \mathbf{1}_{\{n_i=0\}} \\
&\quad + [e^F,(S_i^*-S_i)]\varphi_{\ell_i} - S_ie^F\big( \varphi_{r_i} + \varphi_{\ell_i} \big)\\
&= e^F A_i' g + 2^{-1}e^F(S^*_i+S_i)g  - e^F(S_i^*-S_i) N_i (g_{\ell_i}+g) + e^F(S_i-S_i^*)\varphi_{\ell_i} \mathbf{1}_{\{n_i=0\}} \\
&\quad + [e^F,(S_i^*-S_i)]\varphi_{\ell_i} - S_ie^F\big( \varphi_{r_i} + \varphi_{\ell_i} \big)
\end{align*}
Therefore we have obtained
\begin{equation}
\label{SecondDD}
e^F A'[e^F,\Delta] + [e^F,\Delta]A' e^F =  2e^F A' g A' e^F + e^F (L_{\ell} + M_{\ell} + G_{\ell}+H_{\ell})e^F, \quad \text{where}
\end{equation}
\begin{equation*}
L_{\ell} := -\sum_{i,j} A_i'(g_{\ell_j}+g) N_j (S_j^*-S_j) + (S_i^*-S_i) N_i (g_{\ell_i}+g) A_j',
\end{equation*}
\begin{equation*}
M_{\ell} := 2^{-1}\sum_{i,j} - A_i' g (S_j^*+S_j)  +  (S_i^*+S_i) g A_j',
\end{equation*}
\begin{align*}
G_{\ell} &:= \sum_{i,j} A_i'\left(-\varphi_{\ell_j} [e^F,(S_j^*-S_j)]e^{-F} + \big( \varphi_{r_j} + \varphi_{\ell_j} \big) e^F S_j^* e^{-F} \right) \\
&\quad +\sum_{i,j} \left( e^{-F} [e^F,(S_i^*-S_i)]\varphi_{\ell_i} - e^{-F}S_ie^F\big( \varphi_{r_i} + \varphi_{\ell_i} \big) \right) A_j', \quad \text{and}
\end{align*}
\begin{equation*}
H_{\ell} := \sum_{i,j} A_i'\varphi_{\ell_j} \mathbf{1}_{\{n_j=0\}} (S_j-S_j^*) + (S_i-S_i^*) \varphi_{\ell_i} \mathbf{1}_{\{n_i=0\}} A_j'.  \end{equation*}
We split $M_{\ell}$ as follows: $M_{\ell} := M_{\ell;1}+M_{\ell;2}$, where 
\begin{equation*}
M_{\ell;1} := 2^{-1}\sum_{i \neq j} - A_i' g (S_j^*+S_j)  +  (S_i^*+S_i) g A_j',
\end{equation*}
\begin{equation*}
M_{\ell;2} := 2^{-1}\sum_i - A_i' g (S_i^*+S_i)  +  (S_i^*+S_i) g A_i' = M_{\ell;2;1} + M_{\ell;2;2}, \quad \text{with}
\end{equation*}
\begin{equation*}
M_{\ell;2;1} := 2^{-1}\sum_i  A_i' g_{\ell_i} (S_i^*+S_i) -  (S_i^*+S_i) g_{\ell_i} A_i',
\end{equation*}
\begin{equation*}
M_{\ell;2;2} := 2^{-1}\sum_i  -A_i' (g+g_{\ell_i}) (S_i^*+S_i) +  (S_i^*+S_i) (g+g_{\ell_i}) A_i'.
\end{equation*}
We calculate $M_{\ell;1}$ by expanding $A_i'$ and $A_j'$:
\begin{align*}
M_{\ell;1} &= 2^{-1}\sum_{i \neq j} -N_i(S_i^*-S_i) g(S_j^*+S_j)  +  (S_i^*+S_i) g N_j(S_j^*-S_j) \\
&= 2^{-1}\sum_{i \neq j} -N_i\big[ (\T_i^*g)S_i^*-(\T_ig)S_i\big](S_j^*+S_j) + \big[(\T_i^*(gN_j))S_i^*+(\T_i(gN_j))S_i\big](S^*_j-S_j) \\
&= \frac{1}{2}\sum_{i \neq j} N_i\big[\T_ig-\T_jg\big]S_iS_j + N_i\big[ \T_j^*g-\T_i^*g\big]S_i^*S_j^* + \big[N_i(\T_jg-\T_i^*g)+N_j(\T_jg-\T_i^*g)\big]S_i^*S_j.^{(\ddagger_2)}
\end{align*}
Again expanding $A_i'$:
\begin{align*}
M_{\ell;2;1} &= 2^{-1}\sum_i -(S_i^*+S_i)g_{\ell_i}(S_i^*+S_i) 
+(S_i^*-S_i)\varphi_{\ell_i}\mathbf{1}_{\{n_i \neq 0\}}(S_i^*+S_i) \\
& \quad - 2^{-1}\sum_i (S_i^*+S_i)\varphi_{\ell_i}\mathbf{1}_{\{n_i \neq 0\}}(S_i^*-S_i) \\
&= \sum_i -2^{-1}(S_i^*+S_i)g_{\ell_i}(S_i^*+S_i) + S_i^*\varphi_{\ell_i} S_i - S_i\varphi_{\ell_i} S_i^*   + \left(S_i\varphi_{\ell_i} \mathbf{1}_{\{n_i = 0\}} S_i^* - S_i^*\varphi_{\ell_i} \mathbf{1}_{\{n_i = 0\}} S_i\right) \\
&= M_{\ell;2;1;1}+M_{\ell;2;1;2}, \quad \text{where}
\end{align*}
\begin{equation*}
M_{\ell;2;1;1} := \sum_i -2^{-1}(S_i^*+S_i)g_{\ell_i}(S_i^*+S_i) + (\T_i^*\varphi_{\ell_i} - \T_i\varphi_{\ell_i}),^{(\ddagger_1,\ddagger_3)}
\end{equation*}
\begin{equation*}
M_{\ell;2;1;2} := \sum_i (\T_i\varphi_{\ell_i})\mathbf{1}_{\{n_i=+1\}}-(\T_i^*\varphi_{\ell_i})\mathbf{1}_{\{n_i=-1\}}.
\end{equation*}
We calculate $G_{\ell}$: 
\begin{align*}
G_{\ell} &:= \sum_{i,j} A_i' \left(-\varphi_{\ell_j} \big( S_j^* \varphi_{\ell_j} - S_j \varphi_{r_j} \big) + \big( \varphi_{r_j} + \varphi_{\ell_j} \big) S_j^* \big(\varphi_{\ell_j} + 1 \big) \right) \\
&\quad + \sum_{i,j} \left(\big( -\varphi_{r_i}S_i^* + \varphi_{\ell_i} S_i \big)\varphi_{\ell_i} - \big(\varphi_{\ell_i}+1\big) S_i \big( \varphi_{r_i} + \varphi_{\ell_i} \big) \right) A_j' \\
&= \sum_{i,j} A_i' \left(-S_j^* (\T_j \varphi_{\ell_j})\varphi_{\ell_j} + S_j (\T_j^*\varphi_{\ell_j})\varphi_{r_j} + S_j^* (\T_j\varphi_{r_j}) (\varphi_{\ell_j}+1) + S_j^* (\T_j\varphi_{\ell_j})(\varphi_{\ell_j}+1) \right) \\
&\quad +\sum_{i,j} \left(-\varphi_{r_i}(\T_i^*\varphi_{\ell_i})S_i^* + \varphi_{\ell_i}(\T_i \varphi_{\ell_i})S_i   - (\varphi_{\ell_i}+1)(\T_i \varphi_{r_i})S_i - (\varphi_{\ell_i}+1)(\T_i\varphi_{\ell_i})S_i \right) A_j' \\
&= \sum_{i,j} A_i' \left( S_j^* (\T_j\varphi_{\ell_j}- \varphi_{\ell_j}) - S_j(\T_j^* \varphi_{\ell_j}+\varphi_{r_j}) \right) + \sum_{i,j}\left( (\varphi_{\ell_i}-\T_i\varphi_{\ell_i})S_i + (\T_i^*\varphi_{\ell_i}+ \varphi_{r_i})S_i^* \right) A_j' \\
&= G_{\ell;1} + G_{\ell;2}, \quad \text{where}
\end{align*}
\begin{equation*}
G_{\ell;1} := \sum_{i,j} A_i' S_j^* (\T_j\varphi_{\ell_j}- \varphi_{\ell_j}) +  (\varphi_{\ell_i}-\T_i\varphi_{\ell_i})S_i A_j',^{(\ddagger_3)}
\end{equation*}
\begin{equation*}
\label{GL2}
G_{\ell;2} := \sum_{i,j}  -A_i' S_j(\T_j^* \varphi_{\ell_j}+\varphi_{r_j})  +  (\T_i^*\varphi_{\ell_i}+ \varphi_{r_i})S_i^*A_j'.
\end{equation*}
Note that we are left to deal with $L_{\ell} + M_{\ell;2;1;2} + M_{\ell;2;2} + G_{\ell;2} + H_{\ell}$. \\
\textbf{Part 3 : Adding the terms of Parts 1 and 2.} Take the average of \eqref{FirstDD} and \eqref{SecondDD}:
\begin{equation}
\begin{aligned}
\label{SunHunBun}
[e^F A' e^F,\Delta] &= e^F[A',\Delta]e^F + 2e^F A' g A' e^F \\
& \quad + 2^{-1}e^F\left(L_r + L_{\ell} + M_r+M_{\ell}+G_r+G_{\ell} + H_r + H_{\ell} \right) e^F.
\end{aligned}
\end{equation}
Applying $\phi \in \ell_0(\Z^d)$ to this equation and taking inner products leads to \eqref{theFirstoneDDD}. We go into details. The terms that still have to be dealt with are $L_r + M_{r;2;1;2} + M_{r;2;2} + G_{r;2} + H_r$ from the first part and $L_{\ell} + M_{\ell;2;1;2} + M_{\ell;2;2} + G_{\ell;2} + H_{\ell}$ from the second part. Since
\begin{equation*}
(\T_i^*\varphi_{r_i}-\T^*_i \varphi_{\ell_i})\mathbf{1}_{\{n_i=-1\}}\phi = (\T_i\varphi_{\ell_i}-\T_i\varphi_{r_i})\mathbf{1}_{\{n_i=+1\}} \phi = 0, \quad \text{and} \quad (\varphi_{r_i}-\varphi_{\ell_i}) \mathbf{1}_{\{n_i=0\}} \phi = 0,
\end{equation*}
it follows that 
\begin{equation*}
(M_{r;2;1;2} + M_{\ell;2;1;2})\phi = 0, \quad \text{and} \quad
(H_r + H_{\ell})\phi = 0.
\end{equation*}
We add $L_r$ and $L_{\ell}$ and define this to be $\mathcal{L}$:
\begin{equation}
\begin{aligned}
\label{LL}
\mathcal{L} := L_r + L_{\ell} &= \sum_{i,j} A_i'[(g_{r_j}-g)-(g_{\ell_j}+g)] N_j (S_j^*-S_j) \\
&\quad + \sum_{i,j}(S_i^*-S_i) N_i [(g_{r_i}-g)-(g_{\ell_i}+g)] A_j'.^{(\ddagger_4)}
\end{aligned}
\end{equation}
We add $M_{r;2;2}$ and $M_{\ell;2;2}$:
\begin{equation*}
M_{r;2;2} + M_{\ell;2;2} = 2^{-1}\sum_i A_i' [(g_{r_i}-g)-(g_{\ell_i}+g)] (S_i^*+S_i)  - (S_i^*+S_i) [(g_{r_i}-g)-(g_{\ell_i}+g)] A_i'.^{(\ddagger_4)} 
\end{equation*}
We can now define $\mathcal{M}$: 
\begin{equation}
\label{MM}
\mathcal{M} := M_r + M_{\ell} = M_{r;1} + M_{r;2;1;1} + M_{\ell;1} + M_{\ell;2;1;1} + (M_{r;2;2} + M_{\ell;2;2}). 
\end{equation}
The final step is to add $G_{r;2}$ and $G_{\ell;2}$:
\begin{align*}
G_{r;2}+G_{\ell;2} &= \sum_{i,j} -A_iS_j^*(\varphi_{\ell_j}+\T_j\varphi_{r_j}) + (\varphi_{\ell_i}+\T_i\varphi_{r_i})S_iA_j \\
& \quad +\sum_{i,j} - A_iS_j(\T_j^*\varphi_{\ell_j}+\varphi_{r_j})+(\T_i^*\varphi_{\ell_i}+\varphi_{r_i})S^*_iA_j  \\
&= - \sum_{i,j} [2^{-1}(S_i^*+S_i)+N_i(S_i^*-S_i)]S_j^*(\varphi_{\ell_j}+\T_j\varphi_{r_j}) \\
&\quad + \sum_{i,j} (\varphi_{\ell_i}+\T_i\varphi_{r_i})S_i[-2^{-1}(S_j^*+S_j)+(S_j^*-S_j)N_j] \\
& \quad - \sum_{i,j} [2^{-1}(S_i^*+S_i)+N_i(S_i^*-S_i)] S_j(\T_j^*\varphi_{\ell_j}+\varphi_{r_j}) \\
& \quad + \sum_{i,j} (\T_i^*\varphi_{\ell_i}+\varphi_{r_i})S^*_i[-2^{-1}(S_j^*+S_j)+(S_j^*-S_j)N_j] \\
&= G_1+G_2+G_3+G_4+G_5+G_6, \quad \text{where}
\end{align*}
\begin{equation*}
G_1 := \sum_{i,j} -N_i S_i^*S_j^*(\varphi_{\ell_j}+\T_j\varphi_{r_j}) + (\T_i^*\varphi_{\ell_i}+\varphi_{r_i}) S_i^* S_j^* N_j,
\end{equation*}
\begin{equation*}
G_2 := \sum_{i,j} N_iS_iS_j (\T_j^*\varphi_{\ell_j}+\varphi_{r_j}) - (\varphi_{\ell_i}+\T_i\varphi_{r_i})S_iS_jN_j,
\end{equation*}
\begin{equation*}
G_3 := \sum_{i,j} N_iS_iS_j^* (\varphi_{\ell_j}+\T_j\varphi_{r_j}) - N_iS_i^*S_j (\T^*_j\varphi_{\ell_j}+\varphi_{r_j})  - (\T^*_i \varphi_{\ell_i} + \varphi_{r_i})S_i^*S_j N_j + (\varphi_{\ell_i}+\T_i\varphi_{r_i}) S_i S_j^* N_j,
\end{equation*}
\begin{equation*}
G_4 := -2^{-1} \sum_{i,j}  S_i^*S_j^*(\varphi_{\ell_j}+\T_j\varphi_{r_j}) + (\T^*_i\varphi_{\ell_i}+\varphi_{r_i})S_i^*S_j^*,
\end{equation*}
\begin{equation*}
G_5 := - 2^{-1} \sum_{i,j} (\varphi_{\ell_i}+\T_i\varphi_{r_i})S_iS_j + S_iS_j(\T^*_j\varphi_{\ell_j}+\varphi_{r_j}), 
\end{equation*}
\begin{equation*}
G_6 := - 2^{-1} \sum_{i,j} S_iS_j^*(\varphi_{\ell_j}+\T_j\varphi_{r_j}) + (\varphi_{\ell_i}+\T_i\varphi_{r_i})S_iS_j^*  + S_i^*S_j(\T^*_j\varphi_{\ell_j}+\varphi_{r_j}) + (\T^*_i\varphi_{\ell_i}+\varphi_{r_i})S_i^*S_j. 
\end{equation*}
We calculate $G_i$ for $i=1...6$. $G_1 = G_{1;1} + G_{1;2} +G_{1;3}$, with
\begin{equation*}
G_{1;1} := \sum_{i,j} [(\T_j^*\varphi_{\ell_j}-\T_i^*\T_j^*\varphi_{\ell_j})+(\varphi_{r_j}-\T_i^*\varphi_{r_j})]N_iS_i^*S_j^*,^{(\ddagger_3)}  
\end{equation*}
\begin{equation*}
G_{1;2} := \sum_{i \neq j} (\T_i^*\varphi_{\ell_i}+\varphi_{r_i}) S_i^* S_j^* \quad \text{and} \quad
G_{1;3} := 2 \sum_i (\T_i^*\varphi_{\ell_i}+\varphi_{r_i}) (S_i^*)^2.
\end{equation*}
\begin{equation*}
G_2 = G_{2;1} + G_{2;2} + G_{2;3}, \quad \text{where}
\end{equation*}
\begin{equation*}
G_{2;1} := \sum_{i,j} [(\T_i \varphi_{\ell_j}- \varphi_{\ell_j})+(\T_i \T_j \varphi_{r_j}-\T_j \varphi_{r_j})]N_iS_iS_j,^{(\ddagger_3)}     
\end{equation*}
\begin{equation*}
G_{2;2} := \sum_{i \neq j} (\varphi_{\ell_i}+\T_i\varphi_{r_i})S_iS_j \quad \text{and} \quad
G_{2;3} := 2\sum_i (\varphi_{\ell_i}+\T_i\varphi_{r_i})(S_i)^2.
\end{equation*}
\begin{align*}
G_3 &= \sum_{i \neq j} N_iS_iS_j^* (\varphi_{\ell_j}+\T_j\varphi_{r_j}) - N_iS_i^*S_j (\T^*_j\varphi_{\ell_j}+\varphi_{r_j}) - (\T^*_i \varphi_{\ell_i} + \varphi_{r_i})S_i^*S_j N_j + (\varphi_{\ell_i}+\T_i\varphi_{r_i}) S_i S_j^* N_j \\
&\quad + \sum_i N_i (\varphi_{\ell_i}+\T_i\varphi_{r_i}) - N_i (\T^*_i\varphi_{\ell_i}+\varphi_{r_i}) - (\T^*_i \varphi_{\ell_i} + \varphi_{r_i}) N_i + (\varphi_{\ell_i}+\T_i\varphi_{r_i}) N_i \\
&= \sum_{i \neq j} (\T_i \T_j^* \varphi_{\ell_j}+\T_i\varphi_{r_j})N_iS_iS_j^* - (\T^*_i\varphi_{\ell_j}+\T_i^*\T_j\varphi_{r_j})N_iS_i^*S_j \\
& \quad + \sum_{i \neq j} - (\T^*_i \varphi_{\ell_i} + \varphi_{r_i})N_j S_i^*S_j + (\varphi_{\ell_i}+\T_i\varphi_{r_i}) N_j S_i S_j^* \\
&\quad +\sum_{i \neq j} (\T^*_i \varphi_{\ell_i} + \varphi_{r_i}) S_i^*S_j + (\varphi_{\ell_i}+\T_i\varphi_{r_i})S_i S_j^* + 2\sum_i  [(\varphi_{\ell_i}-\T^*_i \varphi_{\ell_i})+(\T_i\varphi_{r_i}-\varphi_{r_i})]N_i\\
&= G_{3;1}+G_{3;2}+G_{3;3}, \quad \text{where}
\end{align*}
\begin{equation*}
G_{3;1} := \sum_{i \neq j} [(\T_i \T_j^* \varphi_{\ell_j}-\T_j^* \varphi_{\ell_j})+(\T_i\varphi_{r_j}-\varphi_{r_j})]N_iS_iS_j^* +[(\varphi_{\ell_j}-\T^*_i\varphi_{\ell_j})+(\T_j\varphi_{r_j}-\T_i^*\T_j\varphi_{r_j})]N_iS_i^*S_j,^{(\ddagger_3)}    
\end{equation*}
\begin{equation*}
G_{3;2} := \sum_{i \neq j} (\T^*_i \varphi_{\ell_i} + \varphi_{r_i}) S_i^*S_j + (\varphi_{\ell_i}+\T_i\varphi_{r_i})S_i S_j^*,   
\end{equation*}
\begin{equation*}
G_{3;3} :=  2\sum_i  [(\varphi_{\ell_i}-\T^*_i \varphi_{\ell_i})+(\T_i\varphi_{r_i}-\varphi_{r_i})]N_i.^{(\ddagger_3)}     
\end{equation*}
\begin{equation*}
G_4 = -2^{-1} \sum_{i,j} [(\T^*_i\T^*_j\varphi_{\ell_j}+\T_i^*\varphi_{r_j})+(\T^*_i\varphi_{\ell_i}+\varphi_{r_i})]S^*_iS^*_j = G_{4;1} + G_{4;2}, \quad \text{with}
\end{equation*}
\begin{equation*}
G_{4;1} := -2^{-1} \sum_{i \neq j} [(\T^*_i\T^*_j\varphi_{\ell_j}+\T_i^*\varphi_{r_j})+(\T^*_i\varphi_{\ell_i}+\varphi_{r_i})]S^*_iS^*_j,  
\end{equation*}
\begin{equation*}
G_{4;2} := -2^{-1} \sum_i [(\T^*_i\T^*_i\varphi_{\ell_i}+\T_i^*\varphi_{r_i})+(\T^*_i\varphi_{\ell_i}+\varphi_{r_i})](S^*_i)^2.  
\end{equation*}
\begin{equation*}
G_5 = -2^{-1} \sum_{i,j} [(\varphi_{\ell_i}+\T_i\varphi_{r_i})+(\T_i\varphi_{\ell_j}+\T_i\T_j\varphi_{r_j})]S_iS_j = G_{5;1} + G_{5;2}, \quad \text{with}
\end{equation*}
\begin{equation*}
G_{5;1}= -2^{-1} \sum_{i \neq j} [(\varphi_{\ell_i}+\T_i\varphi_{r_i})+(\T_i\varphi_{\ell_j}+\T_i\T_j\varphi_{r_j})]S_iS_j,
\end{equation*}
\begin{equation*}
G_{5;2}= -2^{-1} \sum_i [(\varphi_{\ell_i}+\T_i\varphi_{r_i})+(\T_i\varphi_{\ell_i}+\T_i\T_i\varphi_{r_i})](S_i)^2.
\end{equation*}
\begin{align*}
G_6 &= -2^{-1} \sum_{i,j} [(\T_i\T_j^*\varphi_{\ell_j}+\T_i\varphi_{r_j})+(\varphi_{\ell_i}+\T_i\varphi_{r_i})]S_iS_j^* + [(\T^*_i\varphi_{\ell_j}+\T_i^*\T_j\varphi_{r_j})+(\T^*_i\varphi_{\ell_i}+\varphi_{r_i})]S_i^*S_j \\
&= G_{6;1} +G_{6;2}, \quad \text{with} 
\end{align*}
\begin{equation*}
G_{6;1} := -2^{-1} \sum_{i \neq j} [(\T_i\T_j^*\varphi_{\ell_j}+\T_i\varphi_{r_j})+(\varphi_{\ell_i}+\T_i\varphi_{r_i})]S_iS_j^* + [(\T^*_i\varphi_{\ell_j}+\T_i^*\T_j\varphi_{r_j})+(\T^*_i\varphi_{\ell_i}+\varphi_{r_i})]S_i^*S_j,   
\end{equation*}
\begin{equation*}
G_{6;2} := - \sum_i (\varphi_{\ell_i}+\T_i\varphi_{r_i}) + (\T^*_i\varphi_{\ell_i}+\varphi_{r_i}).   
\end{equation*}
We add $G_{1;2}$ and $G_{4;1}$:
\begin{align*}
G_{1;2}+G_{4;1} &= \sum_{i\neq j} (\T_i^*\varphi_{\ell_i}+\varphi_{r_i})S^*_iS^*_j -2^{-1} \sum_{i\neq j} [(\T^*_i\T^*_j\varphi_{\ell_j}+\T_i^*\varphi_{r_j})+(\T^*_i\varphi_{\ell_i}+\varphi_{r_i})]S^*_iS^*_j  \\
&=2^{-1} \sum_{i\neq j} [(\T^*_j\varphi_{\ell_j}-\T^*_i\T^*_j\varphi_{\ell_j})+(\varphi_{r_j}-\T_i^*\varphi_{r_j})]S^*_iS^*_j.^{(\ddagger_3)}  
\end{align*}
We add $G_{1;3}$ and $G_{4;2}$:
\begin{align*}
G_{1;3}+G_{4;2}&= 2\sum_i (\T_i^*\varphi_{\ell_i}+\varphi_{r_i})(S^*_i)^2-2^{-1} \sum_i [(\T^*_i\T^*_i\varphi_{\ell_i}+\T_i^*\varphi_{r_i})+(\T^*_i\varphi_{\ell_i}+\varphi_{r_i})](S^*_i)^2  \\
&= G_7 + G_8, \quad \text{where}
\end{align*}
\begin{equation*}
G_7 := 2^{-1} \sum_i [(\T^*_i\varphi_{\ell_i}-\T^*_i\T^*_i\varphi_{\ell_i})+(\varphi_{r_i}-\T_i^*\varphi_{r_i})](S^*_i)^2 \ ^{(\ddagger_3)}   \quad \text{and} \quad
G_8 := \sum_i (\T_i^*\varphi_{\ell_i}+\varphi_{r_i})(S^*_i)^2.  
\end{equation*}
We add $G_{2;2}$ and $G_{5;1}$:
\begin{align*}
G_{2;2}+G_{5;1} &= \sum_{i \neq j} (\varphi_{\ell_i}+\T_i\varphi_{r_i})S_iS_j -2^{-1} \sum_{i \neq j} [(\varphi_{\ell_i}+\T_i\varphi_{r_i})+(\T_i\varphi_{\ell_j}+\T_i\T_j\varphi_{r_j})]S_iS_j \\
&= 2^{-1} \sum_{i \neq j} [(\varphi_{\ell_j}-\T_i\varphi_{\ell_j})+(\T_j\varphi_{r_j}-\T_i\T_j\varphi_{r_j})]S_iS_j.^{(\ddagger_3)}  
\end{align*}
We add $G_{2;3}$ and $G_{5;2}$:
\begin{align*}
G_{2;3}+G_{5;2}&= 2\sum_i (\varphi_{\ell_i}+\T_i\varphi_{r_i})(S_i)^2 - 2^{-1} \sum_i [(\varphi_{\ell_i}+\T_i\varphi_{r_i})+(\T_i\varphi_{\ell_i}+\T_i\T_i\varphi_{r_i})](S_i)^2  \\
&= G_9 + G_{10}, \quad \text{where}
\end{align*}
\begin{equation*}
G_9 := 2^{-1} \sum_i [(\varphi_{\ell_i}-\T_i\varphi_{\ell_i})+(\T_i\varphi_{r_i}-\T_i\T_i\varphi_{r_i})](S_i)^2 \ ^{(\ddagger_3)}   \quad \text{and} \quad
G_{10} := \sum_i (\varphi_{\ell_i}+\T_i\varphi_{r_i})(S_i)^2.   
\end{equation*}
We add $G_{3;2}$ and $G_{6;1}$:
\begin{equation*}
G_{3;2}+G_{6;1} = -2^{-1} \sum_{i \neq j} [(\T_i\T_j^*\varphi_{\ell_j}-\T_j^*\varphi_{\ell_j})+(\T_i\varphi_{r_j}-\varphi_{r_j})+(\T^*_j\varphi_{\ell_i}-\varphi_{\ell_i})+(\T_j^*\T_i\varphi_{r_i}-\T_i\varphi_{r_i})]S_iS_j^*.^{(\ddagger_3)}  
\end{equation*}
We are left to deal with $G_{6;2}$, $G_8$ and $G_{10}$:
\begin{align*}
G_{8}+G_{10}+G_{6;2}&= \sum_i (\T_i^*\varphi_{\ell_i}+\varphi_{r_i})S^*_iS^*_i +  (\varphi_{\ell_i}+\T_i\varphi_{r_i})S_iS_i - (\varphi_{\ell_i}+\T_i\varphi_{r_i}) -  (\T^*_i\varphi_{\ell_i}+\varphi_{r_i}) \\
&= \sum_i (\T_i^*\varphi_{\ell_i}+\varphi_{r_i})((S^*_i)^2-1) + (\varphi_{\ell_i}+\T_i\varphi_{r_i})((S_i)^2-1) \\
&= \sum_i [(\T_i^* \varphi_{\ell_i}-\varphi_{\ell_i}) + (\varphi_{r_i}-\T_i \varphi_{r_i})]((S_i^*)^2-1) + ( \varphi_{\ell_i}+ \T_i \varphi_{r_i})((S_i^*)^2+S_i^2-2) \\
&= G_{11} + G_{12}, \quad \text{where}
\end{align*}
\begin{equation*}
G_{11} := \sum_i[(\T_i^* \varphi_{\ell_i}-\varphi_{\ell_i}) + (\varphi_{r_i}-\T_i \varphi_{r_i})]((S_i^*)^2-1)^{(\ddagger_3)}   \ \text{and} \
G_{12} := - 2\sum_i (\cosh(\T_i F-F) -1)\Delta_i(4-\Delta_i).  
\end{equation*}
Let $W_{F;i} := \sqrt{\cosh(\T_i F-F)-1}$. Commuting $W_{F;i}$ with $\Delta_i$ gives
\begin{equation*}
W_{F;i} \Delta_i  = \Delta_i W_{F;i} + S_i \big( W_{F;i}-\T_i^*W_{F;i}\big) + S_i^* \big( W_{F;i}-\T_i W_{F;i}\big).
\end{equation*}
Thus
\begin{equation*}
W_{F;i}^2 \Delta_i(4-\Delta_i) = W_{F;i}\Delta_i(4-\Delta_i)W_{F;i} + R_{F;i}, \quad \text{where}
\end{equation*}
\begin{equation*}
\begin{aligned}
R_{F;i} &:= - W_{F;i}\Delta_i S_i \big(W_{F;i} -\T_i^*W_{F;i}\big) - W_{F;i}\Delta_i S_i^* \big(W_{F;i} -\T_i W_{F;i}\big) \\
& \quad + W_{F;i} S_i \big(W_{F;i} -\T_i^*W_{F;i}\big) (4-\Delta_i) + W_{F;i} S_i^* \big(W_{F;i} -\T_i W_{F;i}\big) (4-\Delta_i)^{\ddagger_5}.
\end{aligned}
\end{equation*}
A final accounting job gives the expression of $\mathcal{G}$:
\begin{equation}
\begin{aligned}
\label{GG}
\mathcal{G} &:= G_{r;1}+G_{\ell;1}+G_{1;1}+G_{2;1}+G_{3;1}+ G_{3;3}+ (G_{1;2}+G_{4;1}) \\
&\quad + G_7 + (G_{2;2}+G_{5;1})+G_9+(G_{3;2}+G_{6;1}) +G_{11} -2\sum_i R_{F;i},   
\end{aligned}
\end{equation}
or equivalently, $\mathcal{G} = G_r + G_{\ell} + 2\sum_i W_{F;i} \Delta_i(4-\Delta_i)W_{F;i}$.
\qed
\end{proof}

\vspace{0.5cm}
\begin{center}
$\star \quad \star \quad \star$
\end{center}
\vspace{0.5cm}
We now turn to the proof of relation \eqref{theFirstoned1} that is key in Proposition \ref{propdodo}. Here $d=1$. For convenience we rewrite the relation we want to show. For $\phi \in \ell_0(\Z)$, $\phi_F := e^F \phi$ :  
\begin{equation}
\begin{aligned}
\label{theFirstone}
\big \langle \phi, [e^F A' e^F,\Delta] \phi \big \rangle &= \big \langle \phi_F, [A', \Delta] \phi_F \big \rangle - \big\|\sqrt{g_r-g_{\ell}}A' \phi_F \big\|^2 \\
& \quad - \big\|\sqrt{\Delta(4-\Delta)} W \phi_F \big\|^2 + 2^{-1} \big \langle \phi_F, (M + G) \phi_F \big \rangle, \quad \text{where} 
\end{aligned}
\end{equation}
\begin{equation}
\label{theOmegadd}
W = W_F := \sqrt{\cosh(\T F-F)-1},
\end{equation}
\begin{equation}
\begin{aligned}
\label{FunnyGuy}
M = M_F &:= 2^{-1}(S^*+S)(g_r-g_{\ell})(S^*+S) \\
& \quad + \big[ (\T^* \varphi_{\ell} - \varphi_{\ell}) + (\varphi_{\ell} - \T \varphi_{\ell}) + (\T\varphi_r - \varphi_r) + (\varphi_r - \T^*\varphi_r) \big], \quad \text{and}
\end{aligned}
\end{equation}
\begin{align}
\begin{split}
\label{goHunt}
G = G_F &:= A'S(\T^*\varphi_r -\varphi_r) + (\varphi_r - \T^*\varphi_r)S^*A' + A'S^*(\T\varphi_{\ell}-\varphi_{\ell}) + (\varphi_{\ell}-\T\varphi_{\ell})S A' \\
& \quad +\big[ (\T^*\varphi_{\ell} - {\T^*}^2 \varphi_{\ell}) + (\varphi_r - \T^*\varphi_r) \big] N{S^*}^2  +  \big[ (\T^2 \varphi_r-\T \varphi_r) + (\T\varphi_{\ell} - \varphi_{\ell}) \big] NS^2 \\
&\quad + \frac{1}{2}\big[ (\T^*\varphi_{\ell}-{\T^*}^2\varphi_{\ell}) + (\varphi_r - \T^* \varphi_r) \big] (S^*)^2  + \frac{1}{2}\big[ (\T \varphi_r - \T^2 \varphi_r) + (\varphi_{\ell}-\T \varphi_{\ell})\big] S^2 \\
&\quad + 2 \big[(\varphi_{\ell} - \T^* \varphi_{\ell}) + (\T \varphi_r - \varphi_r) \big] N + \big[(\T^* \varphi_{\ell}-\varphi_{\ell}) + (\varphi_r-\T \varphi_r)\big]((S^*)^2-1) \\
&\quad + 2 W_F\Delta S \big(W_F -W_{\T^*F}\big) + 2 W_F\Delta S^* \big(W_F -W_{\T F}\big) \\
&\quad - 2 W_F S \big(W_F -W_{\T^* F}\big) (4-\Delta) - 2 W_F S^* \big(W_F -W_{\T F}\big) (4-\Delta).
\end{split} 
\end{align}
\noindent \textit{Proof of} \eqref{theFirstone}. For the most part, the proof of this relation is the same as that of \eqref{theFirstoneDDD} when $d \geqslant 1$. However, the main difference is that here we do not introduce the function $g(n) := F'(\langle n \rangle) / \langle n \rangle$. We go over the proof done just above and point out the differences. As before we start with
\begin{equation*}
[e^F A' e^F, \Delta] = e^F[A', \Delta]e^F + e^FA'[e^F,\Delta] + [e^F,\Delta]A' e^F
\end{equation*}
and develop the last two terms of this relation. \\
\textbf{Part 1 : Creating $e^F A'g_r A' e^F$.}
\begin{equation*}
[e^F, \Delta] = g_r A' e^F - \frac{1}{2}g_r (S^*+S) e^F + \varphi_r \mathbf{1}_{\{n=0\}} (S^*-S) e^F + \varphi_r [e^F,(S^*-S)] + \big( \varphi_r + \varphi_{\ell} \big) e^F S. 
\end{equation*}
\begin{equation*}
[e^F, \Delta] = e^F A' g_r +\frac{1}{2}e^F(S^*+S) g_r + e^F(S^*-S)\varphi_r \mathbf{1}_{\{n=0\}} - [e^F,(S^*-S)]\varphi_r - S^*e^F\big( \varphi_r + \varphi_{\ell} \big). \end{equation*}
Therefore we have obtained
\begin{equation}
\label{Second1dd}
e^F A'[e^F,\Delta] + [e^F,\Delta]A' e^F =  2e^F A' g_r A' e^F + e^F (M_r + G_r + H_r)e^F, \quad \text{where}
\end{equation}
\begin{equation*}
M_r := -2^{-1}A' g_r (S^*+S) + 2^{-1}(S^*+S) g_r A', 
\end{equation*}
\begin{align*}
G_r &:= A' \varphi_r [e^F,(S^*-S)]e^{-F} + A' \big( \varphi_r + \varphi_{\ell} \big) e^F S e^{-F} \\
&- e^{-F} [e^F,(S^*-S)]\varphi_r A'  - e^{-F}S^*e^F\big( \varphi_r + \varphi_{\ell} \big) A', \quad \text{and}
\end{align*}
\begin{equation*}
H_r := A' \varphi_r \mathbf{1}_{\{n=0\}} (S^*-S) + (S^*-S)\varphi_r \mathbf{1}_{\{n=0\}} A'.
\end{equation*}
We calculate $M_r$:
\begin{align*}
M_r &= -\frac{1}{2}\left(-\frac{1}{2}\left(S^*+S\right)+(S^*-S)N\right)g_r(S^*+S)  + \frac{1}{2}\left(S^*+S\right)g_r\left(\frac{1}{2}\left(S^*+S\right)+ N(S^*-S)\right) \\
&= 2^{-1}(S^*+S)g_r(S^*+S) - 2^{-1}\left(S^*-S\right)\varphi_r\mathbf{1}_{\{n \neq 0\}}\left(S^*+S\right) +2^{-1}\left(S^*+S\right)\varphi_r\mathbf{1}_{\{n \neq 0\}}\left(S^*-S\right) \\
&= 2^{-1}(S^*+S)g_r(S^*+S) + \left(S\varphi_r S^* - S^*\varphi_rS\right)   + \left(S^*\varphi_r \mathbf{1}_{\{n=0\}} S - S\varphi_r \mathbf{1}_{\{n=0\}} S^*\right) \\
&= M_{r;1} + M_{r;2}, \quad \text{where}
\end{align*}
\begin{equation*}
M_{r;1} := 2^{-1}(S^*+S)g_r(S^*+S) + \big[ (\T\varphi_r - \varphi_r) + (\varphi_r - \T^*\varphi_r) \big] \ \  \text{and} \ \ M_{r;2} := \varphi_r(0)\left(\mathbf{1}_{\{n=-1\}}-\mathbf{1}_{\{n=1\}}\right).
\end{equation*}
\textbf{Part 2 : Creating $e^F A'g_{\ell} A' e^F$.} 
\begin{equation*}
[e^F, \Delta] = - g_{\ell} A' e^F + \frac{1}{2}g_{\ell} (S^*+S) e^F - \varphi_{\ell} \mathbf{1}_{\{n=0\}} (S^*-S) e^F - \varphi_{\ell} [e^F,(S^*-S)] + \big( \varphi_r+\varphi_{\ell} \big) e^F S^*. 
\end{equation*}
\begin{equation*}
[e^F, \Delta] = -e^F A' g_{\ell} -\frac{1}{2}e^F(S^*+S) g_{\ell} - e^F(S^*-S)\varphi_{\ell} \mathbf{1}_{\{n=0\}}  - [(S^*-S),e^F]\varphi_{\ell} - Se^F\big( \varphi_r + \varphi_{\ell} \big). 
\end{equation*}
Therefore we have obtained
\begin{equation}
\label{First1dd}
e^F A'[e^F,\Delta] + [e^F,\Delta]A' e^F =  -2e^F A' g_{\ell} A' e^F + e^F (M_{\ell} + G_{\ell} + H_{\ell})e^F,\quad \text{where}
\end{equation}
\begin{equation*}
M_{\ell} := 2^{-1} A' g_{\ell} (S^*+S) - 2^{-1}(S^*+S) g_{\ell} A',
\end{equation*}
\begin{align*}
G_{\ell} &:= -A' \varphi_{\ell} [e^F,(S^*-S)]e^{-F} + A' \big( \varphi_r + \varphi_{\ell}  \big) e^F S^*e^{-F} \\
&+ e^{-F} [e^F,(S^*-S)]\varphi_{\ell}A'  - e^{-F}Se^F\big( \varphi_r + \varphi_{\ell} \big) A', \quad \text{and}
\end{align*}
\begin{equation*}
H_{\ell} := -A'\varphi_{\ell} \mathbf{1}_{\{n=0\}} (S^*-S) -(S^*-S)\varphi_{\ell} \mathbf{1}_{\{n=0\}} A'.   
\end{equation*}
We calculate $M_{\ell}$:
\begin{equation*}
M_{\ell} = M_{\ell;1} + M_{\ell;2}, \quad \text{where}
\end{equation*}
\begin{equation*}
M_{\ell;1} := -2^{-1}(S^*+S)g_{\ell}(S^*+S) + \big[ (\T^*\varphi_{\ell} - \varphi_{\ell}) + (\varphi_{\ell} - \T\varphi_{\ell}) \big] \ \ \text{and} \ \ M_{\ell;2} := \varphi_{\ell}(0)\left(\mathbf{1}_{\{n=1\}}-\mathbf{1}_{\{n=-1\}}\right).    
\end{equation*}
\textbf{Part 3 : Adding the terms of Parts 1 and 2.} Take the average of \eqref{Second1dd} and \eqref{First1dd} to get :
\begin{equation*}
[e^F A' e^F,\Delta] = e^F[A',\Delta]e^F + e^F A' (g_r-g_{\ell})A' e^F + 2^{-1}e^F\left(M_r+M_{\ell}+G_r+G_{\ell} +H_{r} +H_{\ell} \right) e^F.
\end{equation*}
Applying $\phi \in \ell_0(\Z)$ to this equation and taking inner products will yield \eqref{theFirstone}. Let us elaborate exactly how this is achieved. First, let
\begin{equation*}
M := M_r + M_{\ell} = M_{r;1} + M_{\ell;1}.
\end{equation*}
The latter equality holds because $(M_{r;2}+M_{\ell;2})\phi = 0$. Second, note that $G_r$, $G_{\ell}$, $H_r$ and $H_{\ell}$ are exactly the same as in the preceding proof when $i=j=1$, which corresponds to $d=1$. These terms are handled in the same way. In particular $(H_r + H_{\ell})\phi = 0$. Finally, we investigate $G$. Referring to the preceding proof with $i=j=1$, let 
\begin{align*}
G &:= G_{r;1} + G_{\ell;1} + G_{1;1} + G_{2;1} + G_{3;3} + G_7 + G_9 + G_{11} -2R_{F;1}.
\end{align*}
Terms that do not contribute here are: $G_{3;1}$, $G_{1;2} + G_{4;1}$, $G_{2;2}+G_{5;1}$, $G_{3;2}+G_{6;1}$. We warn the careful reader that $G$ is not simply $G_r + G_{\ell}$, because somewhere hidden in $G_{r;2}+G_{\ell;2}$ is the term $-2 W \Delta(4-\Delta)W$ which needs to be extracted. After taking inner products, this term ultimately produces $-\|\sqrt{\Delta(4-\Delta)}W\phi_F\|^2$. Alternatively, $G = G_r + G_{\ell} + 2W\Delta(4-\Delta)W$.
\qed
\\
\noindent We also note that 
\begin{equation}
\begin{aligned}
\label{joker}
G_r + G_{\ell} &= G_{r;1} + G_{\ell;1} + G_{1;1} + G_{2;1} + G_{3;3} + G_7 + G_9 + G_{11} + G_{12} \\
&= 2\big[ (\T^*\varphi_{\ell} - {\T^*}^2 \varphi_{\ell}) + (\varphi_r - \T^*\varphi_r) \big] N{S^*}^2  -  2\big[ (\T \varphi_r -\T^2 \varphi_r) + (\varphi_{\ell} -\T \varphi_{\ell}) \big] NS^2\\
& \quad + 2 \big[ (\varphi_{\ell} -\T \varphi_{\ell}) + (\T^* \varphi_r - \varphi_r) + (\varphi_{\ell} - \T^* \varphi_{\ell}) + (\T \varphi_r - \varphi_r) \big] N \\
& \quad + \big[ (\T^*\varphi_{\ell}-{\T^*}^2\varphi_{\ell}) + 2(\varphi_r - \T^* \varphi_r) \big] {S^*}^2  + \big[ (\T \varphi_r - \T^2 \varphi_r) + 2(\varphi_{\ell}-\T \varphi_{\ell})\big] S^2 \\
& \quad + \big[ (\T \varphi_{\ell}-\varphi_{\ell}) + (\T^*\varphi_r-\varphi_r) \big] + \big[(\T^* \varphi_{\ell}-\varphi_{\ell}) + (\varphi_r-\T \varphi_r)\big]({S^*}^2-1) \\
& \quad - 2(\cosh(\T F-F) -1)\Delta(4-\Delta).
\end{aligned}
\end{equation}




\end{document}